\documentclass[11pt]{article}

\usepackage{enumerate}
\usepackage{amsthm}
\usepackage{amsmath}
\usepackage{amssymb}
\usepackage{mathtools}
\usepackage{bbold,bbm}
\usepackage{hyperref}
\hypersetup{
    colorlinks,
    citecolor=red,
    filecolor=black,
    linkcolor=red,
    urlcolor=black
}
\usepackage{stmaryrd}
\usepackage{setspace}

\usepackage{amsthm}

\providecommand{\keywords}
{
	\textbf{\textit{Keywords and phrases:}}
}
\providecommand{\amssubj}
{
	\textbf{\textit{AMS 2010 subject classification:}}
}

\newtheorem{theorem}{Theorem}[section]
\newtheorem{corollary}[theorem]{Corollary}
\newtheorem{lemma}[theorem]{Lemma}
\newtheorem{definition}[theorem]{Definition}
\newtheorem{remark}[theorem]{Remark}
\newtheorem{proposition}[theorem]{Proposition}
\newtheorem{example}[theorem]{Example}

\newcommand{\norm}[1]{\left\lVert#1\right\rVert}
\newcommand{\normm}[1]{{\left\vert\kern-0.25ex\left\vert\kern-0.25ex\left\vert #1 
    \right\vert\kern-0.25ex\right\vert\kern-0.25ex\right\vert}}
\numberwithin{equation}{section}

\DeclareMathOperator{\Borel}{{\mathfrak B}}
\renewcommand{\d}{{\mathrm d}}

\newcommand{\scapro}[2]{\langle #1,#2\rangle}  
\DeclareMathOperator{\R}{{\mathbb R}}

\renewcommand{\Phi}{F}

\allowdisplaybreaks

\newcommand{\1}{\mathbbm{1}} 
\renewcommand{\phi}{\varphi}
\allowdisplaybreaks

\newcommand{\itemEq}[1]{%
	\begingroup%
	\setlength{\abovedisplayskip}{0pt}%
	\setlength{\belowdisplayskip}{0pt}%
	\parbox[c]{\linewidth}{\begin{flalign}#1&&\end{flalign}}%
	\endgroup}

\title{Stochastic integration with respect to cylindrical Lévy processes in Hilbert spaces}
\author{Gergely Bodó \and Markus Riedle}

\author{Gergely Bodó \\ Korteweg-de Vries Institute for Mathematics \\ University of Amsterdam\\	1090 GE Amsterdam \\ The Netherlands \\g.z.bodo@uva.nl\\ 
	\and Markus Riedle\footnote{2nd affiliation: Institute of Mathematical Stochastics, Faculty of Mathematics, TU Dresden,   01062 Dresden, Germany}
	\\Department of Mathematics \\ King's College  \\ London WC2R 2LS\\ United Kingdom\\ \\ markus.riedle@kcl.ac.uk
}

\begin{document}
\maketitle
\begin{abstract}
In this work, we present a comprehensive theory of stochastic integration with respect to arbitrary cylindrical L\'evy processes in Hilbert spaces. Since cylindrical L\'evy processes do not enjoy a semi-martingale decomposition, our approach relies on an alternative approach to stochastic integration by decoupled tangent sequences. The space of deterministic integrands is identified as a modular space described in terms of the characteristics of the cylindrical L\'evy process. The  space of random integrands is described as the space of predictable processes whose trajectories are in the space of deterministic integrands almost surely. The derived space of random integrands is verified as the largest space of potential integrands, based on a classical definition of stochastic integrability. We apply the introduced theory of stochastic integration to establish a dominated convergence theorem. 
\end{abstract}

\begin{flushleft}
	\amssubj{60H05, 60G20, 60G51, 28C20}
	
\end{flushleft}

\section{Introduction}

Cylindrical L\'evy processes serve as a natural generalisation of cylindrical Brownian motion, providing a unified framework for modelling a wide variety of different random perturbations of infinite-dimensional systems. Analogously to cylindrical Brownian motion, these processes generally do not exist as stochastic processes in the usual sense as stochastic processes with values in the underlying infinite-dimensional space, and can only be interpreted in the generalised sense of Gel'fand and Vilenkin \cite{GV} or Segal \cite{S}.

The first systematic treatment of the concept of  cylindrical L\'evy processes in Hilbert and Banach spaces was undertaken by Applebaum and Riedle in their work \cite{applebaum_riedle_2010}. Leveraging the theory of cylindrical measures, as outlined by Badrikian and Chevet \cite{badrikian_chevet_74} and Schwartz \cite{schwartz_81}, the authors established a precise mathematical framework for understanding cylindrical Lévy processes.  Within this framework, cylindrical L\'evy processes has since found applications in modelling random perturbations of partial differential equations such as Bodó et all \cite{b_t_r_2024spdes} and Kosmala and Riedle \cite{KR}. Specific instances of cylindrical Lévy processes as driving noise have also been explored in works by Priola and Zabczyk \cite{priola_zab_11} and Peszat and Zabczyk \cite{pes_zab_13}. However, all  these applications were made under restrictive assumptions, such as additive noise or specific types of noise, owing to the absence of a fully developed general theory of stochastic integration.

The classical approach to stochastic integration relying on a semi-martingale decomposition cannot be applied to cylindrical L\'evy processes since they do not enjoy a semi-martingale decomposition. For this reason,  a stochastic integration theory has been developed almost exclusively only for cylindrical martingales. This has been achieved through a Dol\'eans measure approach by M\'etivier and Pellaumail in \cite{MetivierPellcylindrical}, via the construction of a family of reproducing kernel Hilbert spaces by Mikulevi\v{c}ius and Rozovski\v{\i} in \cite{MikRoz98}, or alternatively by introducing a new type of quadratic variation for cylindrical continuous local martingales in UMD Banach spaces in Veraar and Yaroslavtsev \cite{ver_yar_16}.

In this work, we present a comprehensive theory of stochastic integration for random integrands with respect to arbitrary cylindrical L\'evy processes. The robustness of our developed theory is showcased through establishing a typical dominated convergence theorem. Furthermore, we demonstrate that the class of random  integrands, described by the characteristics of the integrator, constitutes the largest class of potential integrands, based on the classical definition of stochastic integrability introduced by Urbanic and Woyczynski in \cite{urbanik_woyczynski_1967}.

Our approach relies on a two-sided inequality for the metric of convergence in probability in Hilbert spaces applied to the sum of a decoupled tangent sequence representing the stochastic integral for simple integrands. This particular approach to stochastic integration in the real-valued case was pioneered in the late 1980s in a couple of publications by Kwapien and Woyczynski  and  cumulated in their monograph \cite{kwapien_woyczynski_1992}. In the vector-valued setting, similar ideas were simultaneously pursued for stochastic integration in UMD-Banach spaces  by McConnell in \cite{McConnell}. This  work can be viewed as the precursor to the recently developed, comprehensive theory of stochastic integration with respect to cylindrical Brownian motion in UMD spaces by van Neerven, Veraar and Weis. This theory started with their work \cite{Jan-annals} and has led to various novel insights into stochastic partial differential equations, including sharp maximal inequalities. 

The development of a stochastic integration theory for cylindrical L\'evy processes was previously addressed only in the work
Jakubowski and Riedle \cite{jakubowski_riedle_2017}. There, the underlying approach was, similarly as here,  based on studying the sum of the decoupled tangent sequence but only applied to conclude relatively compactness in the Skorokhod space of the uncoupled sum. This approach led only to an integration theory for stochastic processes that are continuous from the left and have  right limits, a restriction often insufficient for practical applications. In particular, this approach could not be extended to include the important case of predictable integrands. Our publication \cite{BR} serves as a precursor to the present work, wherein we tested our methods in the specific case of a canonical $\alpha$-stable cylindrical process as the integrator. This restricted setting significantly simplifies the description of the space of admissible integrands. The current study builds upon this case study, offering a more comprehensive theory of stochastic integration for arbitrary cylindrical Lévy processes, overcoming limitations inherent in previous approaches.

The present work introduces the stochastic integral in two steps: firstly, for deterministic integrands  and secondly, for random integrands. The largest space of deterministic integrands is derived in Theorem  \ref{det_if_and_only_if_integrable} as a modular space described in terms of the characteristics of the cylindrical L\'evy process. The largest space of random integrands is derived in Theorem \ref{pred_iff_integrable}, and can be described as the space of predictable processes whose trajectories are in the space of deterministic integrands almost surely. 

We briefly summarise our article: in Section 2, some fundamental results on cylindrical L\'evy processes and infinitely divisible probability measures in Hilbert spaces are collected.  Section \ref{se.modular-space} introduces the modular space that characterises the space of deterministic integrands, establishing its completeness, metrisability, and linearity. The stochastic integral for deterministic integrands is presented in Section \ref{se.deterministic}. The role of It{\^o}'s isometry for the classical stochastic integral with respect to a Brownian motion is taken by Lemma \ref{le.cont_of_int_op}. Although it does not establish an isometry, it guarantees that a a sequence of deterministic integrands is Cauchy if and only if the corresponding stochastic integrals are Cauchy in the semi-martingale topology. Section \ref{se.stochastic-integrals} sets the foundation for integrating random integrands by introducing necessary definitions and elementary results. The following Section \ref{se.deoupled-tangent} is devoted to the construction of the decoupled tangent sequence. In this section, we also briefly recall relevant  definitions and results on decoupled tangent sequences from \cite{kwapien_woyczynski_1992}. The space of random integrands is derived in Section \ref{se.random-integrands}. Analogous to the deterministic integrands, Corollary \ref{co.small-if-small} establishes a relationship between the Cauchyness in the space of random integrands and the corresponding Cauchyness of stochastic integrals.

\section{Preliminaries}

\subsection{Cylindrical L\'evy processes}

Let $G$ and $H$ be separable Hilbert spaces with inner products $\langle \cdot,\cdot\rangle$ and corresponding norms $\norm{\cdot}$. Let  $(a_k)_{k\in \mathbb{N}}$ and $(b_k)_{k\in \mathbb{N}}$ be orthonormal bases of $G$ and $H$, respectively. We identify the dual of a Hilbert space by the space itself. The Borel $\sigma$-algebra of $H$ is denoted by $\Borel(H)$ and the open unit ball by $B_H:=\{h\in H: \norm{h}<1\}$ and the closed unit ball by $\bar{B}_H:=\{h\in H: \norm{h}\le1\}$.\par
The Banach space of bounded linear operators from $G$ to $H$ will be denoted by $L(G,H)$ with the operator norm $\norm{\cdot}_{G \rightarrow H}$. Its subspace $L_2(G,H)$ of Hilbert-Schmidt operators is endowed with the norm
$\norm{\Phi}_{\rm HS}^2:=\sum_{k=1}^{\infty}\norm{\Phi a_k}^2$ for $\Phi \in L_2(G,H)$.

Let $(\Omega,\Sigma,P)$ be a complete probability space. We will denote by $L_P^0(\Omega,H)$ the space of equivalence classes of measurable functions $X\colon \Omega \rightarrow H$, equipped with the topology of convergence in probability.\par
Let $S$ be a subset of $G$. For each $n \in \mathbb{N}$, elements $g_1,...,g_n \in S$ and Borel sets $A \in \Borel(\mathbb{R}^n)$, we define
\[C(g_1,...,g_n;A):=\{g \in G: (\langle g,g_1 \rangle,...,\langle g,g_n \rangle)\in A\}.\]
Such sets are called cylindrical sets with respect to $A$. The collection of all these cylindrical sets is denoted by  $\mathcal{Z}(G,S)$,  and  it is a $\sigma$-algebra if $S$ is finite and otherwise an algebra. We write 
$\mathcal{Z}(G)$ for $\mathcal{Z}(G,G)$\par
A set function $\mu: \mathcal{Z}(G)\rightarrow [0, \infty]$ is called a cylindrical measure on $\mathcal{Z}(G)$ if for each finite dimensional subset $S \subseteq G$, the restriction of $\mu$ to the $\sigma$-algebra $\mathcal{Z}(G,S)$ is a $\sigma$-additive measure. A cylindrical measure is said to be a cylindrical probability measure if $\mu(G)=1$.\par
A cylindrical random variable $X$ in $G$ is a linear and continuous mapping $X\colon G \rightarrow L_{P}^0(\Omega,\mathbb{R})$. It defines a cylindrical probability measure $\mu_X$ by 
\begin{align*}
	\mu_X\colon  \mathcal{Z}(G) \to [0,1],\qquad
	\mu_X(Z)=P\big( (Xg_1,\dots, Xg_n)\in A\big)
\end{align*}
for cylindrical sets $Z=C(g_1, ... , g_n; A)$. The cylindrical probability measure $\mu_X$ is called the cylindrical distribution of $X$.
We define the characteristic function of the cylindrical random variable $X$ by
\[\varphi_X\colon G \rightarrow \mathbb{C}, \qquad\varphi_X(g)=E\big[e^{iXg}\big].\]
Let $T\colon G\to H$ be a linear and continuous operator. By defining
\begin{align*}
	TX\colon H\to L_{P}^0(\Omega,\mathbb{R}), \qquad (TX)h=X(T^\ast h)
\end{align*}
we obtain a cylindrical random variable on $H$. In the special case when $T$ is a Hilbert-Schmidt operator and hence $0$-Radonifying by \cite[Th.\ VI.5.2]{vakhania_1981}, it follows from  \cite[Pr.\ VI.5.3]{vakhania_1981} that the cylindrical random variable $TX$ is induced by a genuine random variable $Y\colon \Omega\to H$, that is 
$(TX)h =\scapro{Y}{h}$ for all $h\in H$. As shown in \cite[Le. 2.1]{BR}, the inducing random variable $Y$ depends continuously on the Hilbert-Schmidt operator.

A family $(L(t):t\geq 0)$ of cylindrical random variables $L(t)\colon G \rightarrow L_{P}^0(\Omega,\mathbb{R})$ is called a cylindrical Lévy process if for each $n \in \mathbb{N}$ and $g_1,...,g_n \in G$, the stochastic process
\[\big(\big(L(t)g_1,...,L(t)g_n\big): t \geq 0\big)\]
is a Lévy process in $\mathbb{R}^n$. The filtration generated by $(L(t):t\geq 0)$ is defined by
\[\mathcal{F}_t :=\sigma(\{L(s)g\colon g\in G, s\in [0,t]\}) \quad \text{for all}\; t\geq 0.\] 
Denote by $\mathcal{Z}_*(G)$ the collection
\[\big\{\{g \in G: (\langle g,g_1 \rangle,...,\langle g,g_n \rangle)\in B\}:n \in \mathbb{N},g_1,...,g_n \in G,B\in \mathcal{B}(\mathbb{R}^n\setminus\{0\})\big\}\]
of cylindrical sets, which  forms an algebra of subsets of $G$. For fixed $g_1,...,g_n \in G$, let $\lambda_{g_1,\dots, g_n}$ be the L\'evy measure of $((L(t)g_1,...,L(t)g_n ): t \geq 0)$. Define a function 
$\lambda\colon  \mathcal{Z}_*(G) \rightarrow [0,\infty]$ by 
\begin{align*}
\lambda(C):=\lambda_{g_1,\dots, g_n}(B)\quad\text{for }
C=\{g\in G:\, (\langle g,g_1 \rangle,...,\langle g,g_n \rangle)\in B\} \text{ and }B\in \Borel(\R^n).
\end{align*}
It is shown in \cite{applebaum_riedle_2010} that $\lambda$ is well defined. The set function $\lambda$ is called the cylindrical L\'evy measure of $L$.

The characteristic function of a cylindrical L\'evy process $L$ in $G$ takes for each $t\geq 0$  the form
\[\varphi_{L(t)}:G\rightarrow \mathbb{C}, \qquad \varphi_{L(t)}(g)=\exp\left(tS(g)\right),\]
where the mapping $S:G \rightarrow \mathbb{C}$ is called the cylindrical symbol of $L$, and satisfies
\[S(g)=ia(g)-\frac{1}{2}\langle Qg,g \rangle+ \int_G\left(e^{i\langle g,h\rangle}-1-i\langle g,h \rangle \mathbb{1}_{\Bar{B}_\mathbb{R}}(\langle g,h \rangle)\right)\, \lambda({\rm d}h),\]
where $a:G\rightarrow \mathbb{R}$ is a continuous mapping with $a(0)=0$, $Q:G \rightarrow G$ is a positive and symmetric operator, and $\lambda$ is a cylindrical L\'evy measure on $G$. We call the triplet $(a, Q, \lambda)$ the cylindrical characteristics of $L$. For this and related results see \cite{riedle2011infinitely}.

\subsection{Infinitely divisible measures and their characteristics}\label{se.inf_div}

Infinitely divisible measures on a Hilbert space $H$ can be defined as in the Euclidean space; see \cite{parthasarathy_1967}. 
As in finite dimensions, the characteristic function of any infinitely divisible measure $\mu$ on $\Borel(H)$ satisfies
$\phi_\mu(h)=\exp(S(h))$, where the symbol $S\colon H\to {\mathbb C}$ is of the form
\[
S(h)=i\scapro{b^\kappa}{h}-\tfrac12 \scapro{Q h}{h}+\int_H \left(e^{i\scapro{h}{g}}-1-i\scapro{h}{\kappa(g)}\right)\,\lambda(dg),  
\]
where $b^\kappa\in H$, the mapping $Q\colon H\to H$ is nuclear, symmetric and non-negative, the Lévy measure  $\lambda$  is a $\sigma$-finite measure on $\Borel(H)$ satisfying
$\int_H \big(\norm{h}^2\wedge 1\big)\, \lambda({\rm d}h)<\infty$ and $\kappa\colon H\to H$ is a function which is bounded and satisfies $\kappa(h)=h$ in a neighbourhood of $0$. Such a function $\kappa$ is called a truncation function. The triplet $(b^\kappa,Q,\lambda)$ is called characteristics of $\mu$. For different truncation functions  $\kappa$, one obtains the same representation of the symbol $S$ but only the term $b^\kappa$ depends on $\kappa$.
When dealing with limit theorems, for technical reasons, it is often preferable to use a continuous function $\kappa$. A specific example of a continuous truncation function, which will play an important role in the rest of this work, is the truncation function
\[\theta\colon  H \rightarrow H, \qquad\theta(h)= 
\begin{cases}
	h &\text{if}\, \norm{h}\leq 1;\\
	\frac{h}{\norm{h}}             & \text{if}\, \norm{h}>1.
\end{cases}
\]
Let $\kappa\colon H\to H$ be a continuous truncation function. 
A sequence of infinitely divisible measures $\mu_n=(b_n^\kappa,Q_n,\lambda_n)$ with associated sequence $(T_n)_{n \in \mathbb{N}}$ of so-called $S$-operators $T_n\colon H\to H$, which are defined by
\begin{align*}
    \langle T_nh_1,h_2\rangle=\langle Q_n h_1,h_2\rangle + \int_{\norm{u}\leq 1}\langle h_1,u \rangle \langle h_2, u\rangle \, \lambda_n({\rm d}u)\qquad\text{for all }h_1,h_2\in H,
\end{align*}
converges weakly to an infinitely divisible measure $\mu=(b^\kappa,Q,\lambda)$ if and only if the following conditions  hold:
\begin{enumerate}[\hspace{0.5cm}]
	\item \itemEq{\displaystyle (1)\; b^\kappa=\lim_{n\to\infty} b_n^\kappa; \label{eq.cont_first_char}} 
	\item \itemEq{\displaystyle (2)\; \lim_{\delta\downarrow 0} \limsup_{n\rightarrow \infty}\int_{\norm{h}\leq \delta} \langle h,u \rangle^2 \, \lambda_n({\rm d}u)+(Q_n h,h)=(Q h,h) \text{ for all }h\in H;\label{small_jumps_conv}}
	\item \itemEq{\displaystyle (3)\;  \lambda_n \to \lambda \text{ weakly outside of every closed neighbourhood of the origin;}\label{conv_of_levy_measure}}
	\item \itemEq{(4)\; \text{$(T_n)_{n \in \mathbb{N}}$ is compact in the space of nuclear operators.}\label{comp_S_op}}
\end{enumerate}

\begin{remark}\label{re.inf_div_continuity_first_char}
    Let $(\mathcal{I}, \norm{\cdot}_0)$ denote the collection of $H$-valued, infinitely divisible random variables endowed with a translation invariant metric $\norm{\cdot}_0$ generating the topology of convergence in probability. Define the mapping
    \[g:\mathcal{I}\rightarrow H, \qquad g(X)=b^\theta_X,\]
    where $b^\theta_X$ denotes the first characteristic of $X$ with respect to the truncation function $\theta$. Then the function $g$ is continuous according to Equation \eqref{eq.cont_first_char}, and hence, by the topological characterisation of continuity, for all $\epsilon>0$ there exists $\delta>0$, depending only on $\epsilon$ and the metric $\norm{\cdot}_0$, such that $\norm{X}_0<\delta$ implies $\norm{b^\theta_X}<\epsilon$ for all $X \in \mathcal{I}$.
\end{remark}

Let $(\pi_n)_{n \in \mathbb{N}}$ be a sequence of partitions of the interval $[s,t]$ of the form
\[\pi_n=\left\{s=p_{0,n}<p_{1,n}<...<p_{{N(n)},n}=t\right\}.\]
We say that $(\pi_n)_{n \in \mathbb{N}}$ is a nested normal sequence of partitions if:
\begin{enumerate}[\rm (1)]
    \item $\pi_n \subseteq \pi_m$ for all $n\leq m$;
    \item $\displaystyle \lim_{n \rightarrow \infty} \max_{i \in \{1,...,N(n)\}} \left\vert p_{i,n} - p_{{i-1},n} \right\vert=0$.
\end{enumerate}

The following result enables us to express L\'evy characteristics as limits of certain series, which will play a key role in the sequel.
\begin{theorem}\label{th:limit_characteristics}
Let $L$ be an $H$-valued L\'evy process with characteristics $(b^{\theta}, Q, \lambda)$, and let $(\pi_n)_{n \in \mathbb{N}}$ be a nested normal sequence of partitions of $[s,t]$, where for each fixed $n \in \mathbb{N}$ we have $\pi_n=\left\{s=p_{0,n}<p_{1,n}<...<p_{{N(n)},n}=t\right\}$. If we put $d_{i,n}=L(p_{i,n})-L(p_{{i-1},n})$, 
then we have
\begin{enumerate}\label{le.char_as_limit_of_series}
    \item [\rm (1)] $ \displaystyle \lim_{n \rightarrow \infty} \sum_{\pi_n} E\left[\theta(d_{i,n})\right]=(t-s)b^{\theta}$;
    \item [\rm (2)] $\displaystyle \lim_{n \rightarrow \infty} \sum_{\pi_n} E\left[ \norm{d_{i,n}}^2 \wedge 1 \right]= (t-s)\left(\int_H \left( \norm{h}^2 \wedge 1 \right)\, \lambda({\rm d}h)+ {\rm Tr}(Q)\right).$
\end{enumerate}
\end{theorem}
\begin{proof}
For a proof, see \cite[Le. 3.4]{nowak_2003}.
\end{proof}

\section{The modular space}\label{se.modular-space}

Originally introduced by Nakano \cite{nakano_50}, modular spaces serve as natural generalisations of metric spaces. Prominent and non-trivial examples are Orlicz spaces and  generalised Musielak–Orlicz spaces. While numerous different definitions appear in the literature, 
in this work we will always use the following adaption of Nakano's original definition of a generalised modular; see \cite{nakano_68}. Our main objective in this section, apart from defining these spaces, is to establish them as complete, metrisable linear spaces and to demonstrate denseness of simple functions.  Although various abstract conditions on the modular are known in the literature, guaranteeing one of these properties, we found it easier to establish these properties directly.  The metrisability is achieved through recent results on $K$-quasi-metric spaces  in \cite{metrization_2019}.

\begin{definition}\label{def.modular_on_lin_space}
Let $V$ be a real vector space. A function $\Delta:V\rightarrow[0,\infty]$ is called a modular if
\begin{enumerate}[{\rm (1)}]
    \item $\Delta(-v)=\Delta(v)$ for all $v \in V$;
    \item $\displaystyle\inf_{\alpha >0}\Delta(\alpha v)=0$ for all $v \in V$;
    \item $\Delta(\alpha v)\leq \Delta(\beta v)$ for all $0\leq \alpha \leq \beta$ and $v \in V$;
    \item there exists a constant $c>0$ such that
    \[\Delta(v + w)\leq c\left(\Delta(v)+\Delta(w)\right)\quad \text{for all } v,w \in V.\]
\end{enumerate}
\end{definition}
 A function satisfying Condition ${\rm (4)}$ of Definition \ref{def.modular_on_lin_space} is said to be of moderate growth.

As Hilbert-Schmidt operators between Hilbert spaces map cylindrical random variables to genuine random variables, they transform  cylindrical L\'evy processes to genuine L\'evy processes.

\begin{lemma} \label{radonif_of_cyl_levy}
Let $(L(t):t \geq 0)$ be a cylindrical L\'evy process in $G$ with cylindrical characteristics $(a,Q,\lambda)$, and let $\Phi \in L_2(G,H)$ be a Hilbert-Schmidt operator. Then there exists an $H$-valued L\'evy process $(\Phi(L)(t):t \geq 0)$ satisfying  $\langle \Phi(L)(t), h \rangle = L(t)(\Phi^*h)$ for all $t \in [0,T]$ and $h\in H$. Moreover, $F(L)$ has characteristics $(b_{\Phi},FQF^*,\lambda \circ \Phi^{-1})$, where for all $u \in H$
\begin{equation*}
    \langle b_{\Phi},u \rangle= a(\Phi^*u)+\int_H \langle h,u \rangle \big(\mathbb{1}_{B_H}(h)-\mathbb{1}_{B_{\mathbb{R}}}(\langle h,u \rangle)\big)\,(\lambda\circ \Phi^{-1})({\rm d}h).
\end{equation*}
\end{lemma}
\begin{proof}
Existence of the $H$-valued L\'evy process $F(L)$ follows from \cite[Th.A]{radonif_by_single}. To derive the characteristics, first apply \cite[Le.\ 5.4]{riedle2014ornsteinuhlenbeck} to obtain the cylindrical characteristics of $F(L)$, and then use \cite[Le.\ 5.8]{riedle2014ornsteinuhlenbeck} to convert the cylindrical characteristics into genuine characteristics.
\end{proof}

\begin{remark}
In the Lemma above and throughout this article, we use $\lambda\circ F^{-1}$ to denote the classical L\'evy measure of the genuine $H$-valued L\'evy process $F(L)$. In particular, $\lambda\circ F^{-1}$ is a $\sigma$-finite measure on the Borel $\sigma$-algebra $\Borel(H)$. 
Formally, one obtains this $\sigma$-finite measure by extending the image cylindrical measure of $\lambda$ under the mapping $F$ from the cylindrical algebra   $\mathcal{Z}_*(H) $ to the Borel $\sigma$-algebra $\Borel(H)$.
\end{remark}

\begin{remark}
Note that in the special case, when the truncation function is $\theta$, see Section \ref{se.inf_div} of the Preliminaries, the first characteristic $b_\Phi^\theta$ satisfies for all $u \in H$ that
\begin{equation}\label{eq.genuine_char_conversion}
    \langle b_{\Phi}^\theta,u \rangle= a(\Phi^*u)+\int_H \big( \langle \theta(h),u \rangle - \langle h,u \rangle \mathbb{1}_{B_{\mathbb{R}}}(\langle h,u \rangle)\big) \,\left(\lambda\circ \Phi^{-1}\right)({\rm d}h).
\end{equation}
\end{remark}

For the rest of this chapter, we fix a cylindrical L\'evy process $L$ with cylindrical characteristics $(a,Q,\lambda)$. 
\begin{definition}\label{def.k_and_l}
Let $L$ be a  cylindrical L\'evy process  with cylindrical characteristics $(a,Q,\lambda)$ and define
functions $k_L,l_L:L_2(G,H) \rightarrow \mathbb{R}$ by
\begin{align*}
    k_L(\Phi)&= \int_H\left(\norm{h}^2 \wedge 1\right)\left(\lambda \circ \Phi^{-1}\right)({\rm d}h)+{\rm Tr}(FQF^*);\\
    l_L(\Phi)&= \sup_{O \in \Bar{B}_{L(H)}} \norm{b_{O\Phi}^\theta},
\end{align*}
where $\Bar{B}_{L(H)}$ denotes the collection of bounded linear operators $O\colon H \rightarrow H$ satisfying $\norm{O}_{H \rightarrow H} \leq 1$, and the expression  $b_{O\Phi}^\theta$ denotes the first characteristic of the Radonified L\'evy process $OF(L)$  for each $O \in \Bar{B}_{L(H)}$ as defined in Equation \eqref{eq.genuine_char_conversion}.
\end{definition}

\begin{remark}\label{re.monotonicity_of_k_l}
    It follows from the very definitions of $k_L$ and $l_L$ that for each fixed $F \in L_2(G,H)$ and for all $0\leq\alpha \leq \beta$ we have that $k_L(\alpha F)\leq k_L(\beta F)$ and $l_L(\alpha F)\leq l_L(\beta F)$. This observation will be repeatedly used in the sequel.
\end{remark}

\begin{definition}\label{def.modular}
For a measurable function $\psi:[0,T]\rightarrow L_2(G,H)$  define
\begin{align*}
    &m_L'(\psi):=\int_0^T \big( k_L(\psi(t))+l_L(\psi(t))\big)\,{\rm d}t;\\
    &m''(\psi):=\int_0^T \left(\norm{\psi(t)}_{L_2(G,H)}^2\wedge 1\right)\,{\rm d}t;\\
    &m_L(\psi):=m_L'(\psi)+m''(\psi).
\end{align*}
We denote by $\mathcal{M}_{{\rm det},L}^{\rm HS}:=\mathcal{M}_{{\rm det},L}^{\rm HS}(G,H)$ the space of Lebesgue a.e.\ equivalence classes of measurable functions $\psi:[0,T]\rightarrow L_2(G,H)$ for which $m_L(\psi)< \infty$.
\end{definition}

\begin{example}
A  cylindrical L\'evy process $L$ on $G$ is called standard symmetric $\alpha$-stable if its characteristic function is of the form \[
\phi_{L(t)}\colon G\to {\mathbb C}, \qquad \phi_{L(t)}(g)=\exp(-\norm{g}^\alpha),
\]
for some $\alpha \in (0,2)$, in which case the theory simplifies significantly. In particular, since $L$ has cylindrical characteristics $(0,0,\lambda)$, we have that $l_L$ and the second term of $k_L$ are always zero. Moreover, the first term of $k_L$ can be equivalently controlled by an $L^\alpha$ norm, that is, by \cite[Le. 3.5.]{BR}, \cite[Le. 3.1.]{KR}  and \cite[Le. 2.1.]{b_t_r_2024spdes}, there exist constants $c_\alpha,d_\alpha>0$ such that for all measurable functions $\psi:[0,T]\rightarrow L_2(G,H)$ we have
\begin{align*}
    \frac{1}{c_\alpha}\int_0^T \norm{\psi(t)}_{L_2(G,H)}^{\alpha}\, {\rm d}t &\leq \int_0^T\int_H \big(\norm{h}^2\wedge 1\big)\, (\lambda \circ \psi(t)^{-1}) ({\rm d}h) \, {\rm d} t\nonumber\\
    &\leq d_\alpha \int_0^T \norm{\psi(t)}_{L_2(G,H)}^{\alpha}\, {\rm d}t.
\end{align*}
Consequently, we obtain that $\mathcal{M}_{{\rm det},L}^{\rm HS}=L_{\rm Leb}^\alpha([0,T],L_2(G,H))$.
\end{example}

The rest of this section will be devoted to proving that $\mathcal{M}_{{\rm det},L}^{\rm HS}$ is a vector space and $m_L$ is a modular on $\mathcal{M}_{{\rm det},L}^{\rm HS}$ in the sense of Definition \ref{def.modular_on_lin_space}. One of the preliminary result, Lemma  \ref{le.properties_k_l}, guarantees that  the integrals in the above definition are well defined. 
As a first step towards this direction, the next lemma provides us with an alternative representation of $l_L$. This will be heavily used in the sequel when we investigate various  properties of the modular.

\begin{lemma}\label{le.alternative_form_of_l}
Let $L$ be a cylindrical L\'evy process in $G$ with characteristics $(a,Q,\lambda)$. For all $\Phi \in L_2(G,H)$ and $O \in L(H)$ it holds 
\begin{equation*}
    b_{O\Phi}^\theta= Ob_\Phi^\theta+ \int_H \big( \theta(Oh)-O \theta(h)\big) \, (\lambda \circ \Phi^{-1})({\rm d}h).
\end{equation*}
\end{lemma}

\begin{proof}
The term $b_{O\Phi}^\theta$ must coincide with the corresponding term in the characteristics of the L\'evy process which we obtain as the image of the L\'evy process $F(L)$ under the map $O$. Using this observation, the formula follows immediately from \cite[Pr. 11.10.]{sato}.
\end{proof}

\begin{remark}\label{re.finite_l}
    It follows from Lemma \ref{le.alternative_form_of_l} that  $l_L(F)$ is finite for each $F \in L_2(G,H)$. To see this, we first note that for all $h \in \Bar{B}_H$ and $O \in \Bar{B}_{L(H)}$ we have $\norm{\theta(Oh)-O \theta(h)}=0$, and for all $h \in H$ and $O \in \Bar{B}_{L(H)}$ it holds that $\norm{\theta(Oh)-O \theta(h)}\leq 2$. By combining these observations with Lemma \ref{le.alternative_form_of_l} we obtain
    \begin{align*}
        \sup_{O \in \Bar{B}_{L(H)}}\norm{b_{O\Phi}^\theta}
        &\leq  \norm{b_F^\theta}+ 2 (\lambda \circ \Phi^{-1})(\Bar{B}_H^c)<\infty.
    \end{align*}
\end{remark}

Before we could prove that our modular $m_L$ is well-defined, we need to establish a relationship between weak convergence of infinitely divisible measures and convergence of the corresponding characteristics in the following sense:

\begin{lemma} \label{le.measure_theoretic_lemma}
Let $\mu_n\stackrel{\mathcal{D}}{=}(b_n^{\theta},Q_n,\lambda_n)$ be a sequence of infinitely divisible measures on $\Borel(H)$ converging weakly to $\mu\stackrel{\mathcal{D}}{=}(b^{\theta},Q,\lambda)$. Then the following conditions hold:
\begin{enumerate}
    \item [\rm (1)] $ \displaystyle \lim_{n\rightarrow \infty}\left(\int_H\left(\norm{h}^2 \wedge 1\right) \; \lambda_n\,({\rm d}h)+ {\rm Tr}(Q_n)\right)=\int_H\left(\norm{h}^2 \wedge 1\right) \; \lambda\,({\rm d}h)+{\rm Tr}(Q)$;
    \item [\rm (2)] $\displaystyle \lim_{n \rightarrow \infty} \norm{b_n^\theta - b^\theta}=0$.
\end{enumerate}
\end{lemma}

\begin{proof}
The fact that $({\rm2})$ holds follows directly from Equation \eqref{eq.cont_first_char}. To prove $({\rm 1})$, fix $\delta \in (0,1]$ such that $\delta \in C(\lambda)$. By Equation \eqref{conv_of_levy_measure} we have
\[\lim_{n\rightarrow \infty}\int_{\norm{h}>\delta}\Big(\norm{h}^2 \wedge 1 \Big)\; \lambda_n({\rm d}h)=\int_{\norm{h}>\delta}\Big(\norm{h}^2 \wedge 1 \Big)\; \lambda({\rm d}h).\]
Therefore, it remains only to deal with the limit of the integrals over $\Bar{B}_H(\delta)$. Let $\epsilon>0$ be fixed. It follows from properties of the Lebesgue integral that there exists a $\delta_1 \in (0,\delta]$ such that
\begin{align} \label{EQAN2}
    \int_{\norm{h}\leq \delta_1}\norm{h}^2\; \lambda({\rm d}h)<\frac{\epsilon}{12}.
\end{align}
Let $\{e_k\}_{k \in \mathbb{N}}$ be an orthonormal basis of $H$. Since $Q$ is a trace class operator, there exists $K_1 \in \mathbb{N}$ such that
\[\sum_{k={K_1+1}}^{\infty}\langle Qe_k,e_k \rangle<\frac{\epsilon}{12}.\]
By compactness of the associated $S$-operators, see Condition \eqref{comp_S_op}, there exists $K_2 \in \mathbb{N}$ such that for all $n \in \mathbb{N}$
\begin{equation} \label{EQN3}
    \sum_{k={K_2+1}}^{\infty} \left(\int_{\norm{h}\leq \delta}\langle e_k,h \rangle^2 \;\lambda_n({\rm d}h)+\langle Q_ne_k,e_k\rangle\right)<\frac{\epsilon}{4}.
\end{equation}
Moreover, by an application of Condition \eqref{small_jumps_conv}, there exists a $\delta_2<\delta_1$ and $N_1 \in \mathbb{N}$ such that for all $n\geq N_1$ and for all $k \leq K_2$ we have that
\begin{equation} \label{EQN4}
    \left \vert\int_{\norm{h}\leq \delta_2}\langle e_k,h\rangle^2 \;\lambda_n({\rm d}h)+\langle Q_ne_k,e_k \rangle - \langle Qe_k, e_k \rangle \right \vert< \frac{\epsilon}{12K},
\end{equation}
where $K:=\max\{K_1,K_2\}$. Condition \eqref{conv_of_levy_measure} guarantees that there exists $N_2 \in \mathbb{N}$ such that for all $n\geq N_2$ we have 
\begin{align} \label{EQN5}
    \Bigg \vert \int_{\delta_2<\norm{h}\leq \delta} \norm{h}^2 \;\lambda_n({\rm d}h) - \int_{\delta_2<\norm{h}\leq \delta} \norm{h}^2\; \lambda({\rm d}h) \Bigg \vert< \frac{\epsilon}{2}.
\end{align}
By splitting the integration domain, we obtain 
\begin{align} \label{EQN6}
      \Bigg \vert &\int_{\norm{h}\leq \delta} \norm{h}^2 \;\lambda({\rm d}h) + {\rm Tr}(Q) - \int_{\norm{h}\leq \delta} \norm{h}^2\; \lambda_n({\rm d}h)-{\rm Tr}(Q_n) \Bigg \vert \nonumber \\
      & \leq\Bigg \vert \int_{\delta_2<\norm{h}\leq \delta} \norm{h}^2 \;\lambda({\rm d}h)-\int_{\delta_2<\norm{h}\leq \delta} \norm{h}^2\; \lambda_n({\rm d}h)\Bigg \vert \nonumber\\
      & \qquad + \Bigg \vert \int_{\norm{h}\leq \delta_2} \norm{h}^2\; \lambda({\rm d}h) + {\rm Tr}(Q) - \int_{\norm{h}\leq \delta_2} \norm{h}^2\; \lambda_n({\rm d}h)- {\rm Tr}(Q_n) \Bigg \vert.
\end{align}
By Parseval's identity, Equations (\ref{EQAN2})-(\ref{EQN4}) and a repeated application of the triangle inequality, we obtain for all $n\ge N:=\max\{N_1,N_2\}$  that
\begin{align}
   &\Bigg \vert \int_{\norm{h}\leq \delta_2} \norm{h}^2 \;\lambda({\rm d}h) + {\rm Tr}(Q) - \int_{\norm{h}\leq \delta_2} \norm{h}^2 \;\lambda_n({\rm d}h)- {\rm Tr}(Q_n) \Bigg \vert \nonumber\\
   \leq& \Bigg \vert {\rm Tr}(Q) - \int_{\norm{h}\leq \delta_2} \norm{h}^2 \;\lambda_n({\rm d}h)- {\rm Tr}(Q_n) \Bigg \vert + \left \vert \int_{\norm{h}\leq \delta_2} \norm{h}^2 \;\lambda({\rm d}h)\right \vert \nonumber\\
   \le & \left \vert \sum_{k=1}^K \left( \langle Qe_k,e_k \rangle - \int_{\norm{h}\leq \delta_2} \langle h,e_k \rangle^2 \;\lambda_n({\rm d}h)- \langle Q_ne_k,e_k \rangle\right) \right \vert \nonumber\\
   &+\left \vert \sum_{k=K+1}^\infty \left( \langle Qe_k,e_k \rangle - \int_{\norm{h}\leq \delta_2} \langle h,e_k \rangle^2 \;\lambda_n({\rm d}h)- \langle Q_ne_k,e_k \rangle\right) \right \vert + \left \vert \int_{\norm{h}\leq \delta_2} \norm{h}^2 \;\lambda({\rm d}h)\right \vert \nonumber\\
   \le &\frac{\epsilon}{2}. \label{eq.final_estimate}
\end{align}
Hence, if $n\geq N$ then Equations (\ref{EQN5})-(\ref{eq.final_estimate}) together imply

\[\Bigg \vert \int_{\norm{h}\leq \delta} \norm{h}^2 \;\lambda({\rm d}h) + {\rm Tr}(Q) - \int_{\norm{h}\leq \delta} \norm{h}^2 \;\lambda_n({\rm d}h)-{\rm Tr}(Q_n) \Bigg \vert< \epsilon.\]

\noindent Since $\epsilon>0$ was arbitrary, the result follows.
\end{proof}

\begin{lemma}\label{le.properties_k_l}
Let $k_L,l_L:L_2(G,H) \rightarrow \mathbb{R}$ be as in Definition \ref{def.k_and_l}. Then we have:
\begin{itemize}
    \item [{\rm (1)}] $k_L$ is continuous;
    \item [{\rm (2)}] $l_L$ is lower-semicontinuous and continuous at $0$.
\end{itemize}
\end{lemma}
\begin{proof}
Continuity of $k_L$ follows immediately from \cite[Le. 2.1]{BR} and Lemma \ref{le.measure_theoretic_lemma}. To prove that $l_L$ is lower-semicontinuous, we fix $F \in L_2(G,H)$ and a sequence $(F_n)_{n \in \mathbb{N}}\subseteq L_2(G,H)$ satisfying  $\lim_{n \rightarrow \infty}\norm{F_n-F}_{L_2(G,H)}=0$. Let $\epsilon>0$ be fixed. It follows from Remark \ref{re.finite_l} and the very definition of the supremum, that  there exists $O_\epsilon \in \Bar{B}_{L(H)}$ such that $\sup_{O \in \Bar{B}_{L(H)}} \norm{b_{O\Phi}^\theta}\leq \norm{b_{O_\epsilon\Phi}^\theta}+\epsilon$. Since  $\lim_{n \rightarrow \infty}\norm{b_{O_\epsilon F_n}^\theta}=\norm{b_{O_\epsilon F}^\theta}$ by \cite[Le. 2.1]{BR} and Equation \eqref{eq.cont_first_char},  we obtain
\begin{align*}
    \sup_{O \in \Bar{B}_{L(H)}} \norm{b_{O\Phi}^\theta}\leq \norm{b_{O_\epsilon\Phi}^\theta}+\epsilon &= \lim_{n \rightarrow \infty} \norm{b_{O_\epsilon\Phi_n}^\theta} +\epsilon
    \leq \liminf_{n \rightarrow \infty} \sup_{O \in \Bar{B}_{L(H)}} \norm{b_{O\Phi_n}^\theta} +\epsilon.
\end{align*}
As $\epsilon>0$ is arbitrary, the above shows $l_L(F)\leq\liminf_{n \rightarrow \infty}l_L(F_n)$
which proves lower-semicontinuity of $l_L$. 

To show continuity of $l_L$ at $0$, note that by \cite[Le. 2.1]{BR} and Remark \ref{re.inf_div_continuity_first_char}, for all $\epsilon>0$ there exists $\delta>0$ such that $\norm{F}_{L_2(G,H)}\leq \delta$ implies $\norm{b_F^\theta}\leq \epsilon$. Since  $\norm{OF}_{L_2(G,H)}\leq \norm{F}_{L_2(G,H)}$ for all $O \in \Bar{B}_{L(H)}$, we conclude that 
$\norm{F}_{L_2(G,H)}\leq \delta$  implies $\sup_{O \in \Bar{B}_{L(H)}}\norm{b_{OF}^\theta}\leq \epsilon$, 
which concludes the proof.
\end{proof}

In preparation for showing that $m_L$ is of moderate growth, see Definition \ref{def.modular_on_lin_space}/(4), we prove the following technical lemmata.

\begin{lemma} \label{le.HS_approximation}
Let $\{e_i\}_{i \in \mathbb{N}}$ be an orthonormal basis of $G$ and let $P_n\colon G \rightarrow G$ be the projection onto  $\text{Span}\{e_1,...,e_n\}$.
Then we have for all $\Phi \in L_2(G,H)$ that
\[\lim_{n \rightarrow \infty}\norm{\Phi P_n-\Phi}_{L_2(G,H)}=0.\]
\end{lemma}
\begin{proof}
Since $P_n e_i=e_i$ for $i\leq n$, and $P_n e_i=0$ for $i>n$, we have
\[\|\Phi P_n - \Phi\|_{L_2(G,H)}^2=\sum_{i=1}^\infty \|(\Phi P_n - \Phi) e_i\|_H^2 = \sum_{i=n+1}^\infty \|\Phi e_i\|_H^2 \to 0 \text{ as }n\to\infty, \]
by the Hilbert-Schmidt property of $\Phi$.
\end{proof}

\begin{lemma}\label{le.properties_of_k}
For all $\Phi,\Phi_1,\Phi_2 \in L_2(G,H)$ we have
\begin{enumerate}
    \item [\rm (1)] $k_L(\Phi_1+\Phi_2)\leq 2 \left( k_L(\Phi_1)+k_L(\Phi_2) \right)$;
    \item [\rm (2)] $\displaystyle \sup_{O \in \Bar{B}_{L(H)}}k_L(O\Phi)\leq k_L(\Phi)$.
\end{enumerate}
\end{lemma}

\begin{proof}
Let $P_n\colon G \rightarrow G$ denote the projections from Lemma \ref{le.HS_approximation}. Using the inequality
\begin{equation}\label{eq.useful_ineq}
    (a+b)^2 \wedge 1 \leq 2\left[(a^2 \wedge 1)+(b^2 \wedge 1)\right] \; \text{for all} \; a,b \in \mathbb{R},
\end{equation}
we observe  for each $n \in \mathbb{N}$ that
\begin{align} \label{EQN9}
    &\int_H \left(\norm{h}^2 \wedge 1\right) \; \left(\lambda\circ \left( (\Phi_1+\Phi_2)P_n \right)^{-1}\right)({\rm d}h) \\
    =&\int_G \left(\norm{(\Phi_1+\Phi_2)g}^2 \wedge 1\right) \; \left(\lambda \circ P_n^{-1}\right) ({\rm d}g) \nonumber\\
    \leq& 2 \left( \int_G \left(\norm{\Phi_1g}^2 \wedge 1\right) \; \left(\lambda \circ P_n^{-1}\right) ({\rm d}g) +\int_G \left(\norm{\Phi_2g}^2 \wedge 1\right) \; \left(\lambda \circ P_n^{-1}\right) ({\rm d}g)\right)\nonumber\\
    =& 2\left(\int_H \left(\norm{h}^2 \wedge 1\right) \; \left(\lambda\circ \left(\Phi_1 P_n \right)^{-1}\right)({\rm d}h)+\int_H \left(\norm{h}^2 \wedge 1\right) \; \left(\lambda\circ \left(\Phi_2 P_n \right)^{-1}\right)({\rm d}h)\right).\nonumber
\end{align}
Moreover, by symmetry and positivity of $Q$, and the very definition of the Hilbert-Schmidt inner product $\langle A,B \rangle_{L_2(G,H)}={\rm Tr}(AB^*)$ we obtain
\begin{align}\label{eq.trace_calc}
    &{\rm Tr}\left(((\Phi_1+\Phi_2)P_n)Q((\Phi_1+\Phi_2)P_n)^*\right)\nonumber\\
    &\qquad = \norm{(F_1+F_2)P_n Q^{1/2}}_{L_2(G,H)}^2\nonumber\\
    &\qquad\leq\, 2\left(\norm{\Phi_1P_n Q^{1/2}}_{L_2(G,H)}^2+\norm{\Phi_2P_n Q^{1/2}}_{L_2(G,H)}^2\right)\nonumber\\
    &\qquad =\, 2\left({\rm Tr}((\Phi_1P_n) Q (\Phi_1P_n)^*)+{\rm Tr}((\Phi_2P_n) Q (\Phi_2P_n)^*)\right).
\end{align}
By adding the Inequalities in \eqref{EQN9} and \eqref{eq.trace_calc} we get
\[k_L((\Phi_1+\Phi_2)P_n)\leq 2 \left(k_L(\Phi_1P_n)+k_L(\Phi_2P_n)\right).\]
By taking limits on both sides, and using continuity of $k_L$, see Lemma \ref{le.properties_k_l}/(1), the first part of this Lemma is proved.

To prove the second part, we fix $F \in L_2(G,H)$ and obtain for all $O \in \Bar{B}_{L(H)}$ and $n \in \mathbb{N}$ that
\begin{align}\label{eq.ineq_n_approx_1}
    \int_H \left(\norm{h}^2 \wedge 1\right) \; \left(\lambda\circ \left( O\Phi P_n \right)^{-1}\right)({\rm d}h)&=\int_G \left(\norm{(O\Phi)g}^2 \wedge 1\right) \; \left(\lambda \circ P_n^{-1}\right) ({\rm d}g)\nonumber\\
    &\leq \int_G \left(\norm{\Phi g}^2 \wedge 1\right) \; \left(\lambda \circ P_n^{-1}\right) ({\rm d}g)\nonumber\\
    &=\int_H \left(\norm{h}^2 \wedge 1\right) \; \left(\lambda\circ \left( \Phi P_n \right)^{-1}\right)({\rm d}h).
\end{align}
Moreover, using the relationship between the Hilbert-Schmidt norm and the trace operator, we obtain for all $O \in \Bar{B}_{L(H)}$ and $n \in \mathbb{N}$ that
\begin{align}\label{eq.ineq_n_approx_2}
    {\rm Tr}((O\Phi P_n) Q (O\Phi P_n)^*) 
    &= \norm{O\Phi P_n Q^{1/2}}_{L_2(G,H)}^2\nonumber\\ &\leq \norm{\Phi P_n Q^{1/2}}_{L_2(G,H)}^2 = {\rm Tr}((\Phi P_n) Q (\Phi P_n)^*).
\end{align}
By adding Inequalities \eqref{eq.ineq_n_approx_1} and \eqref{eq.ineq_n_approx_2}, and taking limits on both sides, the result follows from Lemmata \ref{le.measure_theoretic_lemma} and \ref{le.HS_approximation}.
\end{proof}

\begin{lemma}\label{le.modular_triangle}
For all $\psi_1,\psi_2 \in \mathcal{M}_{{\rm det},L}^{\rm HS}$ we have
\[m_L(\psi_1+\psi_2)\leq 4\left(m_L\left(\psi_1\right)+m_L\left(\psi_2\right)\right).\]
\end{lemma}

\begin{proof}
Let $\Phi_1,\Phi_2 \in L_2(G,H)$ and $(\pi_n)_{n \in \mathbb{N}}$ be a nested normal sequence of partitions $\pi_n=(t_{i,n})_{i=1,\dots, N(n)}$ of the interval $[0,1]$. In order to simplify the notation,  we define for each $n \in \mathbb{N}$ and $i \in \left\{1,...,N(n)\right\}$
\[A_{i,n}:=F_1(L)(t_{i,n})-F_1(L)(t_{i-1,n})\ \text{and}\ B_{i,n}:=F_2(L)(t_{i,n})-F_2(L)(t_{i-1,n}).\]
Since $(F_1+F_2)(L)=F_1(L)+F_2(L)$, Theorem \ref{th:limit_characteristics} implies 
\begin{align}\label{eq.triangle}
&\norm{b_{\Phi_1+\Phi_2}^{\theta}} \\
=&\lim_{n \rightarrow \infty}\sum_{i=1}^{N(n)} E\left[\theta\left((\Phi_1+\Phi_2)(L)(t_{i,n})-(\Phi_1+\Phi_2)(L)(t_{i-1,n})\right)\right] \nonumber\\
=&\lim_{n \rightarrow \infty}\norm{\sum_{i=1}^{N(n)} E\left[\theta(A_{i,n}+B_{i,n})\right]}\nonumber\\
    \leq&\lim_{n \rightarrow \infty}\left(\norm{\sum_{i=1}^{N(n)} E\left[\theta(A_{i,n}+B_{i,n})-\theta(A_{i,n})-\theta(B_{i,n})\right]}+\norm{\sum_{i=1}^{N(n)} E\left[\theta(A_{i,n})+\theta(B_{i,n})\right]}\right).  \nonumber
\end{align}
Applying the inequality
\[\norm{\theta(h_1+h_2)-\theta(h_1)-\theta(h_2)}\leq 2\left(\theta(\norm{h_1})^2+\theta(\norm{h_2})^2\right) \quad \text{for all}\; h_1,h_2\in H,\]
let us conclude that 
\begin{align}
&\norm{\sum_{i=1}^{N(n)} E\left[\theta(A_{i,n}+B_{i,n})-\theta(A_{i,n})-\theta(B_{i,n})\right]} \nonumber\\
&\qquad     \le2\sum_{i=1}^{N(n)} E\left[\theta(\norm{A_{i,n}})^2\right]+2\sum_{i=1}^{N(n)} E\left[\theta(\norm{B_{i,n}})^2\right] \label{EQN2}.
\end{align}
Inequality \eqref{eq.triangle} together with  the triangle inequality imply
\eqref{eq.triangle} that 
\begin{align*}
    \norm{b_{\Phi_1+\Phi_2}^{\theta}}\leq& \lim_{n \rightarrow \infty}\Bigg(2\sum_{i=1}^{N(n)} E\left[\theta(\norm{A_{i,n}})^2\right]+2\sum_{i=1}^{N(n)} E\left[\theta(\norm{B_{i,n}})^2\right]\\
    &\qquad\qquad+\norm{\sum_{i=1}^{N(n)} E\left[\theta(A_{i,n})\right]}+\norm{\sum_{i=1}^{N(n)} E\left[\theta(B_{i,n})\right]}\Bigg).
\end{align*}
By taking the limit as $n \rightarrow \infty$ and using the limit characterisation of L\'evy characteristics from Theorem \ref{th:limit_characteristics}, we obtain
\begin{align*}
    \norm{b_{\Phi_1+\Phi_2}^{\theta}}&\leq \;2\left( \int_H\left(\norm{h}^2 \wedge 1\right)\left(\lambda \circ \Phi_1^{-1}\right)({\rm d}h)+{\rm Tr}(F_1QF_1^*)\right)\nonumber\\
    &\qquad \qquad+ \;2\left( \int_H\left(\norm{h}^2 \wedge 1\right)\left(\lambda \circ \Phi_2^{-1}\right)({\rm d}h)+{\rm Tr}(F_2QF_2^*)\right)\nonumber\\
    &\qquad \qquad+\norm{b_{\Phi_1}^{\theta}}+\norm{b_{\Phi_2}^{\theta}}\nonumber\\
    &= 2\left(k_L(F_1)+k_L(F_2)\right)+\norm{b_{\Phi_1}^{\theta}}+\norm{b_{\Phi_2}^{\theta}}.
\end{align*}
Taking supremum and using Lemma \ref{le.properties_of_k}/(2) imply 
\begin{align*}
    l_L(\Phi_1 + \Phi_2)&= \sup_{O \in \Bar{B}_{L(H)}} \norm{b_{O(\Phi_1+\Phi_2)}^{\theta}} \leq 2\left(k_L(\Phi_1)+k_L(\Phi_2)\right)+l_L(\Phi_1)+l_L(\Phi_2), 
\end{align*}
which let us conclude from Lemma \ref{le.properties_of_k}/(1) that
\begin{align} \label{eq.ineq_for_k+l}
    k_L(\Phi_1+\Phi_2)+l_L(\Phi_1+\Phi_2)\leq 4\left(k_L(F_1)+k_L(F_2)+l_L(F_1)+l_L(F_2)\right).
\end{align}
Inequality \eqref{eq.useful_ineq} implies for all measurable functions $\psi_1,\psi_2\in \mathcal{M}_{{\rm det},L}^{\rm HS}$ that
\begin{align*}
    &m_L(\psi_1+\psi_2)\\
    &=\int_0^T k_L\big(\psi_1(t)+\psi_2(t)\big)+l_L\big(\psi_1(t)+\psi_2(t)\big)\,{\rm d}t+\int_0^T\left(\norm{\psi_1(t)+\psi_2(t)}_{L_2(G,H)}^2\wedge 1\right)\,{\rm d}t\\
    &\leq 4 \left(\int_0^T k_L(\psi_1(t)+l_L(\psi_1(t)\, {\rm d}t+\int_0^T k_L(\psi_2(t)+l_L(\psi_2(t)\, {\rm d}t\right)\\
    &\qquad \qquad \qquad \qquad + 2 \left(\int_0^T\left(\norm{\psi_1(t)}_{L_2(G,H)}^2\wedge 1\right)\,{\rm d}t+\int_0^T\left(\norm{\psi_2(t)}_{L_2(G,H)}^2\wedge 1\right)\,{\rm d}t\right)\\
    &\leq 4 \left(m_L(\psi_1)+m_L(\psi_2)\right), 
\end{align*}
which completes the proof.
\end{proof}

\begin{lemma}\label{le.bound_on_sup_k_l}
For all $r>0$ there exists $c_r>0$ such that
\[\sup_{\norm{\Phi}_{L_2(G,H)}\leq r}\left(k_L(\Phi)+l_L(\Phi)\right)\leq c_r.\]
\end{lemma}

\begin{proof}
By Lemma \ref{le.properties_k_l}, $k_L+l_L$ is continuous at $0$, from which it follows that there exists a $\delta>0$ such that $\norm{\Phi}_{L_2(G,H)}\leq\delta$ implies $(k_L+l_L)(\Phi)\leq1$. Let $r>0$ be fixed. If we choose $N_r \in \mathbb{N}$ to be large enough so that $\frac{r}{N_r}\leq\delta$, then by a repeated use of Equation \eqref{eq.ineq_for_k+l}, we obtain for some $c_r>0$ that
\begin{align*}
\sup_{\norm{\Phi}_{L_2(G,H)}\leq r}(k_L+l_L)(\Phi)&=\sup_{\norm{\Phi}_{L_2(G,H)}\leq r}(k_L+l_L)\left(N_r\frac{\Phi}{N_r}\right)\\
    &\leq c_r \sup_{\norm{\Phi}_{L_2(G,H)}\leq r}(k_L+l_L)\left(\frac{\Phi}{N_r} \right)\leq c_r,
\end{align*}
which completes the proof.
\end{proof}
\begin{remark}\label{re.bounded_element}
    Lemma \ref{le.bound_on_sup_k_l} guarantees that every bounded functions $\psi:[0,T]\rightarrow L_2(G,H)$ is  in $  \mathcal{M}_{{\rm det},L}^{\rm HS}$. Indeed, if $\sup_{t \in [0,T]}\norm{\psi(t)}_{L_2(G,H)}\leq r$ for some $r>0$, then Lemma \ref{le.bound_on_sup_k_l} implies that there exists $c_r>0$ such that  $\sup_{\norm{\Phi}_{L_2(G,H)}\leq r}\left(k_L(\Phi)+l_L(\Phi)\right)\leq c_r$. Hence we obtain
    \[m_L'(\psi):= \int_0^T \Big(k_L(\psi(t))+l_L(\psi(t))\Big)\,{\rm d}t\leq T\,c_r<\infty.\]
    Since obviously $m''(\psi)< \infty$, it follows  $m_L(\psi)<\infty$.
\end{remark}

Having developed all the technical tools, we now present the main result of this section: 

\begin{theorem}\label{th.modular_moderate_growth}
$\mathcal{M}_{{\rm det},L}^{\rm HS}$ is a linear space and $m_L$ is a modular on $\mathcal{M}_{{\rm det},L}^{\rm HS}$.
\end{theorem}

\begin{proof}
 Lemma \ref{le.modular_triangle} shows for  $\psi_1,\psi_2 \in \mathcal{M}_{{\rm det},L}^{\rm HS}$ that 
\[m_L(\psi_1+\psi_2)\leq 4\left(m_L(\psi_1)+m_L(\psi_2)\right)<\infty,\]
which implies that $\mathcal{M}_{{\rm det},L}^{\rm HS}$ is closed under addition. A similar argument as in Lemma \ref{le.bound_on_sup_k_l} shows that $\mathcal{M}_{{\rm det},L}^{\rm HS}$ is closed under multiplication by scalars, which completes the proof that $\mathcal{M}_{{\rm det},L}^{\rm HS}$ is a vector space. Hence, it remains only to show that $m_L$ satisfies the conditions of Definition \ref{def.modular_on_lin_space}. It follows directly from the definition of $m_L$ that $m_L(-\psi)=m_L(\psi)$ for all $\psi \in \mathcal{M}_{{\rm det},L}^{\rm HS}$. Condition ${\rm (2)}$ of Definition \ref{def.modular_on_lin_space} is a consequence of Lemma \ref{le.properties_k_l}, Remark \ref{re.monotonicity_of_k_l} and Lebesgue's dominated convergence theorem. Condition ${\rm (3)}$ of Definition \ref{def.modular_on_lin_space} follows from an argument similar to Lemma \ref{le.properties_of_k}/(2) and the very definition of $l_L$. Finally,  Condition ${\rm (4)}$ of Definition \ref{def.modular_on_lin_space} is a direct consequence of Lemma \ref{le.modular_triangle}.
\end{proof}

\begin{remark}\label{re.mod_top}
A sequence $(\psi_n)_{n \in \mathbb{N}}\subseteq \mathcal{M}_{{\rm det},L}^{\rm HS}$  is said to converge to some $\psi \in \mathcal{M}_{{\rm det},L}^{\rm HS}$ in the modular $m_L$ if we have $\lim_{n \rightarrow \infty}m_L(\psi_n-\psi)=0$. Since $m_L(\psi)=0$ if and only if $\psi(t)=0$ for Lebesgue almost all $t \in [0,T]$, we have that limits of sequences in the modular are Lebesgue a.e.\ uniquely determined. For this and further properties of modular convergence, see Section 2 of Nakano \cite{nakano_68}. 
\end{remark}

Later on, we will be interested in the space $L_P^0(\Omega,\mathcal{M}_{{\rm det},L}^{\rm HS})$ of $\mathcal{M}_{{\rm det},L}^{\rm HS}$-valued random elements, which we verify as a Polish space in the sequel.

\begin{lemma}\label{le.completeness_modular_top}
The modular space $(\mathcal{M}_{{\rm det},L}^{\rm HS}, m_L)$ is complete, that is, each modular Cauchy sequence on $\mathcal{M}_{{\rm det},L}^{\rm HS}$ is modular convergent.
\end{lemma}

\begin{proof}
Let $(\psi_i)_{i \in \mathbb{N}}\subseteq \mathcal{M}_{{\rm det},L}^{\rm HS}$ be such that $\lim_{i,j \rightarrow \infty}m_L(\psi_i-\psi_j)=0$. Then, for all $\epsilon\in (0,1)$ we have by Markov's inequality that
\begin{align*}
    &\lim_{i,j \rightarrow \infty}{\rm Leb}\left(t \in [0,T]:\norm{\psi_i(t)-\psi_j(t)}_{L_2(G,H)}> \epsilon\right)\nonumber\\
    &\leq \lim_{i,j \rightarrow \infty}\frac{1}{\epsilon^2} \int_0^T \left(\norm{\psi_i(t)-\psi_j(t)}_{L_2(G,H)}^2\wedge 1\right)\,{\rm d}t \leq \lim_{i,j \rightarrow \infty}\frac{1}{\epsilon^2}m_L(\psi_i-\psi_j)=0,
\end{align*}
which implies that the sequence $(\psi_i)_{i \in \mathbb{N}}$ is Cauchy in Lebesgue measure. Hence, there exists a subsequence $(\psi_{i_n})_{n \in \mathbb{N}}$ converging Lebesgue almost everywhere to a measurable function $\psi:[0,T]\rightarrow L_2(G,H)$.

Let $\epsilon>0$ be fixed. By assumption, there exists $N \in \mathbb{N}$ such that for all $i,j \geq N$ we have $m_L(\psi_i-\psi_j)<\epsilon/2$. Since by Lemma \ref{le.properties_k_l}, $k_L$ is continuous and  $l_L$ is lower-semicontinuous, Fatou's lemma implies for all $i \geq N$ that
\begin{align}\label{eq.cacuhy_estimate_1}
    m_L'(\psi_i-\psi)=&\int_0^T (k_L+l_L)(\psi_i(t)-\psi(t))\,{\rm d}t \nonumber\\
    \leq& \int_0^T \liminf_{n \rightarrow \infty}(k_L+l_L)(\psi_i(t)-\psi_{i_n}(t))\,{\rm d}t \\
    \leq & \liminf_{n \rightarrow \infty}\int_0^T (k_L+l_L)(\psi_i(t)-\psi_{i_n}(t))\,{\rm d}t \leq \liminf_{n \rightarrow \infty} m_L(\psi_i-\psi_{i_n})<\frac{\epsilon}{2}.\nonumber
\end{align}
Since $(\psi_{i_n})_{n \in \mathbb{N}}$ converges Lebesgue a.e.\ to $\psi$, using the dominated convergence theorem we obtain
\begin{align}\label{eq.cacuhy_estimate_3}
    m''(\psi_i-\psi)&=\int_0^T \left(\norm{\psi_i(t)-\psi(t)}_{L_2(G,H)}^2\wedge 1\right)\,{\rm d}t \nonumber\\
    &= \int_0^T \lim_{n \rightarrow \infty}\left(\norm{\psi_i(t)-\psi_{i_n}(t)}_{L_2(G,H)}^2\wedge 1\right)\,{\rm d}t \\
    &=\lim_{n \rightarrow \infty}\int_0^T\left(\norm{\psi_i(t)-\psi_{i_n}(t)}_{L_2(G,H)}^2\wedge 1\right)\,{\rm d}t \leq \lim_{n \rightarrow \infty} m_L(\psi_i-\psi_{i_n})<\frac{\epsilon}{2}\nonumber.
\end{align}
Equations (\ref{eq.cacuhy_estimate_1}) and (\ref{eq.cacuhy_estimate_3}) establish that $\psi_i$ converge to $\psi$ in the modular topology. 
Finally, to see that $\psi \in \mathcal{M}_{{\rm det},L}^{\rm HS}$, fix $i_0 \in \mathbb{N}$ such that $m_L(\psi_{i_0}-\psi)\leq 1$. 
It follows that 
\begin{align*}
    m_L(\psi)\leq 4\left(m_L(\psi-\psi_{i_0})+m_L(\psi_{i_0})\right)\leq 4\left(1+m_L(\psi_{i_0})\right)<\infty.
\end{align*}
 which concludes the proof.
\end{proof}

\begin{remark}
Note that Lemma \ref{le.completeness_modular_top} explains the role of $m''$ in the modular $m_L$. In particular, $m''$ is needed to establish completeness of the modular space $(\mathcal{M}_{{\rm det},L}^{\rm HS}, m_L)$ by allowing the identification of a potential $m_L$-limit of an $m_L$-Cauchy sequence.
\end{remark}

Our next goal is to establish that step functions are dense in the modular space $(\mathcal{M}_{{\rm det},L}^{\rm HS}, m_L)$. In particular, this will immediately yield that the modular space is separable.

\begin{lemma}\label{le.density_of_simple_fn}
The collection of Hilbert-Schmidt operator-valued step functions of the form
\begin{align*}
		\psi \colon [0,T]\rightarrow L_2(G,H),\qquad \psi(t)=F_0\mathbb{1}_{\{0\}}(t)+\sum_{i=1}^{n-1} F_i \mathbb{1}_{(t_i,t_{i+1}]}(t),
	\end{align*}
  where $0=t_1<\cdots < t_n=T$, $F_i \in L_2(G,H)$ for each $i \in \{0,...,n-1\}$, is dense in $(\mathcal{M}_{{\rm det},L}^{\rm HS}, m_L)$. Moreover, the modular space $(\mathcal{M}_{{\rm det},L}^{\rm HS}, m_L)$ is separable.
\end{lemma}
\begin{proof}
First, it follows from Remark \ref{re.bounded_element} that each step function of the above form is an element of the modular space $\mathcal{M}_{{\rm det},L}^{\rm HS}$. To prove the claimed result, we first assume that $\psi \in \mathcal{M}_{{\rm det},L}^{\rm HS}$ is bounded, that is, there exists a constant $r>0$ such that $\sup_{t \in[0,T]}\norm{\psi(t)}_{L_2(G,H)}\leq r$. By \cite[Le.\ 1.2.19]{hytonen_neerven_2016}, there exists a sequence $(\psi_n)_{n \in \mathbb{N}}$ of step functions satisfying:
\begin{enumerate}
    \item [\rm (1)] $\sup_{n \in \mathbb{N}}\sup_{t \in [0,T]}\norm{\psi_n(t)}_{L_2(G,H)}\leq r$;
    \item [\rm (2)] $(\psi_n)_{n \in \mathbb{N}}$ converges to $\psi$ Lebesgue a.e.
\end{enumerate}
Since $\sup_{n \in \mathbb{N}}\sup_{t \in [0,T]}\norm{\psi_n(t)-\psi(t)}_{L_2(G,H)}\leq 2r$, 
 Lemma \ref{le.bound_on_sup_k_l} guarantees  that there exists a constant $c>0$ such that
\begin{equation}\label{eq_bound_for_leb}
    \sup_{n \in \mathbb{N}}\sup_{t \in [0,T]}(k_L+l_L)(\psi_n(t)-\psi(t))\leq c.
\end{equation}
Lebesgue's dominated convergence theorem and Lemma \ref{le.properties_k_l} imply that 
\[\lim_{n \rightarrow \infty}\int_0^T (k_L+l_L)(\psi_n(t)-\psi(t))\,{\rm d}t=\int_0^T \lim_{n \rightarrow \infty}(k_L+l_L)(\psi_n(t)-\psi(t))\,{\rm d}t=0.\]
Applying Lebesgue's dominated convergence theorem to $m^{\prime\prime}(\psi_n-\psi)$ shows that the step functions $\psi_n$ converge to $\psi$ in the modular $m_L$. 

In the general case of an arbitrary $\psi$ in $ \mathcal{M}_{{\rm det},L}^{\rm HS}$, we define a sequence of functions
\[\psi_n:[0,T] \rightarrow L_2(G,H),\qquad
    \psi_n(t)= 
\begin{cases}
    \psi(t) & \text{if } \norm{\psi(t)}_{L_2(G,H)}\leq n , \\
    0              & \text{otherwise.}
\end{cases}
\]
It follows from the very definition of $\psi_n$ that for every $n \in \mathbb{N}$ and $t \in [0,T]$ we have
\[(k_L+l_L)(\psi_{n+1}(t)-\psi(t))\leq (k_L+l_L)(\psi_{n}(t)-\psi(t))\leq (k_L+l_L)(\psi(t)).\]
Since $m_L(\psi)<\infty$ we get
\begin{align*}
    \int_0^T (k_L+l_L)(\psi_1(t)-\psi(t))\,{\rm d}t &\leq \int_0^T (k_L+l_L)(\psi(t))\,{\rm d}t\leq m_L(\psi)<\infty.
\end{align*}
The monotone convergence theorem implies 
\begin{align}\label{eq.density_estimate_1}
    \lim_{n \rightarrow \infty}\int_0^T (k_L+l_L)(\psi_n(t)-\psi(t))\,{\rm d}t=0.
\end{align}
Since Lebesgue's dominated convergence theorem shows $m^{\prime\prime}(\psi_n-\psi)\to 0$, 
we obtain $m_L(\psi_n-\psi)\to 0$. By the first part of this lemma, for each $n \in \mathbb{N}$, there exists a sequence $(\psi_{n,i})_{i \in \mathbb{N}}$ of step functions converging to $\psi_n$ in the modular $m_L$ as $i\to\infty$. For each $n \in \mathbb{N}$ we can choose $i_n\in {\mathbb N}$ such that $m_L(\psi_n-\psi_{n,i_n}) < \frac{1}{n}$. It follows from Lemma \ref{le.modular_triangle} that
\[\lim_{n \rightarrow \infty}m_L(\psi-\psi_{n,i_n})\leq \lim_{n \rightarrow \infty} 4\left(m_L(\psi-\psi_n)+m_L(\psi_n-\psi_{n,i_n})\right)=0.\]
Since one might require that the approximating sequence of step functions are defined on rational partitions of the time domain and, by separability of $L_2(G,H)$, only take values in a countable dense subset of $L_2(G,H)$, separability of $(\mathcal{M}_{{\rm det},L}^{\rm HS}, m_L)$ follows.
\end{proof}

\begin{proposition}\label{pr.m_polish}
There exists a translation invariant metric $\rho_L$ on  $\mathcal{M}_{{\rm det},L}^{\rm HS}$ satisfying:
\begin{enumerate}
    \item [{\rm(1)}] $(\mathcal{M}_{{\rm det},L}^{\rm HS}, \rho_L)$ is a Polish space;
    \item [{\rm(2)}] for any sequence $(\psi_n)_{n \in \mathbb{N}}\subseteq \mathcal{M}_{{\rm det},L}^{\rm HS}$ and $\psi \in \mathcal{M}_{{\rm det},L}^{\rm HS}$ we have
\[\lim_{n \rightarrow \infty}m_L(\psi_n-\psi)=0 \iff \lim_{n \rightarrow \infty}\rho_L(\psi_n,\psi)=0.\]
\end{enumerate}
\end{proposition}

\begin{proof}
    It follows from Lemma \ref{le.modular_triangle} and basic properties of the modular $m_L$ that the mapping $(\psi_1,\psi_2)\mapsto m_L(\psi_1-\psi_2)$ defines a $K$-quasi-metric on $\mathcal{M}_{{\rm det},L}^{\rm HS}$ in the sense of \cite[Def. 2.2.]{metrization_2019}. Hence, by \cite[Th. 3.10.]{metrization_2019}, there exists a metric $d_L$ on $\mathcal{M}_{{\rm det},L}^{\rm HS}$ and $p\in (0,1)$ such that
\begin{align}\label{eq.strongly_equivalent_metric}
    d_L(\psi_1,\psi_2)\leq m_L(\psi_1-\psi_2)^p\leq 2\, d_L(\psi_1,\psi_2) \quad \text{for all }  \psi_1,\psi_1 \in \mathcal{M}_{{\rm det},L}^{\rm HS}.
\end{align}
Combining Equation (\ref{eq.strongly_equivalent_metric}) with Lemma \ref{le.completeness_modular_top}, Theorem \ref{th.modular_moderate_growth} and Lemma \ref{le.density_of_simple_fn} we obtain that $(\mathcal{M}_{{\rm det},L}^{\rm HS},d_L)$ is a complete and separable metric linear space. Thus, \cite[Cor. 2.6]{klee_52} implies that there exists a translation invariant metric $\rho_L$, equivalent to $d_L$, such that $(\mathcal{M}_{{\rm det},L}^{\rm HS},\rho_L)$ is a Polish space.
\end{proof}

\section{Stochastic integrals with deterministic integrands}\label{se.deterministic}

 The definition of the stochastic integral for deterministic integrands with respect to a cylindrical L\'evy process $L$ depends heavily on two classes of step functions. We give  a precise definition of what is meant by a step function in the following. 

\begin{definition}\label{def:step_fn}\hfill
\begin{enumerate}
  \item[{\rm (1)}] An $L_2(G,H)$-valued step function is of the form
    \begin{align} \label{eq.det-step-HS}
		\psi \colon [0,T]\rightarrow L_2(G,H),\qquad \psi(t)=F_0\mathbb{1}_{\{0\}}(t)+\sum_{i=1}^{n-1} F_i \mathbb{1}_{(t_i,t_{i+1}]}(t),
	\end{align}
  where $0=t_1<\cdots < t_n=T$, $F_i \in L_2(G,H)$ for each $i \in \{0,...,n-1\}$. The space of $L_2(G,H)$-valued step functions is denoted by $\mathcal{S}_{\rm det}^{\rm HS}:=\mathcal{S}_{\rm det}^{\rm HS}(G,H)$.
  \item[{\rm (2)}] An $L(H)$-valued step function is of the form
  \begin{align}
		\gamma \colon [0,T]\rightarrow L(H),\qquad \gamma(t)=F_0\mathbb{1}_{\{0\}}(t)+\sum_{i=1}^{n-1} F_i \mathbb{1}_{(t_i,t_{i+1}]}(t),
	\end{align}
   where $0=t_1<\cdots < t_n=T$ and $F_i \in L(H)$ for each $i \in \{0,...,n-1\}$. The space of $L(H)$-valued step functions with $\sup_{t\in [0,T]}\norm{\gamma(t)}_{H \rightarrow H}\leq 1$ is denoted by $\mathcal{S}^{\rm 1,op}_{\rm det}:=\mathcal{S}^{\rm 1,op}_{\rm det}(H,H)$.
\end{enumerate}
\end{definition}

Let $L(t_{i+1})-L(t_i)$ be an increment of the cylindrical Lévy process $L$ and assume that $F_i \in L_2(G,H)$ for each $i \in \{1,...,n-1\}$. Since Hilbert-Schmidt operators are $0$-Radonifying by \cite[Th.\ VI.5.2]{vakhania_1981}, it follows from  \cite[Pr.\ VI.5.3]{vakhania_1981} that there exist genuine random variables $\Phi_i\big(L(t_{i+1})-L(t_i)\big)\colon \Omega\to H$  for each $i \in \{1,...,n-1\}$  satisfying
\[(L(t_{i+1})-L(t_i))(\Phi_i^*h)=\langle \Phi_i(L(t_{i+1})-L(t_i)),h \rangle \quad \text{$P$-a.s.\ for all } h \in H.\]
We call the random variables  $\Phi_i\big(L(t_{i+1})-L(t_i)\big)$ Radonified increments for each $i \in \{1,...,n-1\}$ . 
The stochastic integral is defined for any $\psi \in \mathcal{S}_{\rm det}^{\rm HS}$ with representation (\ref{eq.det-step-HS}) as the sum of the Radonified increments 
\[I(\psi):=\int_0^T \psi \, \d L =\sum_{i=1}^{n-1} \Phi_i (L(t_{i+1})-L(t_i)).\]
Thus, the integral $I(\psi):\Omega\rightarrow H$ is a genuine $H$-valued random variable.

The following definition of the stochastic integral can be traced back to the theory of vector measures, and was adapted to the probabilistic setting in \cite{urbanik_woyczynski_1967} by Urbanik and Woyczy{\'n}ski.

\begin{definition} \label{det_integrability}
A function $\psi\colon [0,T]\rightarrow L_2(G,H)$ is $L$-integrable for a given cylindrical L\'evy process $L$ on $G$
 if there exists a sequence $(\psi_n)_{n \in \mathbb{N}}$ of elements of $\mathcal{S}_{\rm det}^{\rm HS}$ satisfying
\begin{enumerate}[\rm(1)]
    \item $(\psi_n)_{n \in \mathbb{N}}$ converges to $\psi$ Lebesgue a.e.;  \label{det_int_def_1}
    \item $\displaystyle \lim_{m,n \rightarrow \infty}\sup_{\gamma \in \mathcal{S}^{\rm 1,op}_{\rm det}}E\Bigg[\norm{\int_0^T \gamma(\psi_m-\psi_n) \, \d L}\wedge 1 \Bigg]=0.$ \label{det_int_def_2}
\end{enumerate}
In this case,  the stochastic integral of the deterministic function $\psi$ is defined by
\[I(\psi):=\int_0^T \psi \, \d L:= \lim_{n\rightarrow \infty} \int_0^T \psi_n \, \d L \quad \text{in}\; L_P^0(\Omega,H).\]
The class of all deterministic $L$-integrable Hilbert-Schmidt operator-valued functions is denoted by $\mathcal{I}_{{\rm det},L}^{\rm HS}:=\mathcal{I}_{{\rm det},L}^{\rm HS}(G,H)$. 
\end{definition}

\begin{remark}\label{cadlag_remark}
If Conditions \eqref{det_int_def_1} and \eqref{det_int_def_2} in Definition \ref{det_integrability} are satisfied, then completeness of $L_P^0(\Omega,H)$ implies the existence of the limit. Furthermore, it follows that the integral process $(\int_0^t \psi \, \d L)_{t \geq 0}$, defined by $\int_0^t \psi\, \d L:=\int_0^T\1_{[0,t]}\psi\, \d L$ has c\'adl\'ag paths. To see this, note that for each $m,n \in \mathbb{N}$ the process $(\int_0^t \left(\psi_m - \psi_n \right)\, \d L)_{t \geq 0}$ has c\'adl\'ag paths. By an extension of \cite[Pr.\ 8.2.1]{kwapien_woyczynski_1992} to $H$-valued processes and Condition (\ref{det_int_def_2}) above, we obtain
\begin{multline*}
    \lim_{m,n \rightarrow \infty}P\Bigg(\sup_{0\leq t\leq T} \norm{\int_0^t \left(\psi_m - \psi_n\right)\, \d L} >\epsilon\Bigg)\\\leq 3 \lim_{m,n \rightarrow \infty} \sup_{0\leq t\leq T} P\Bigg( \norm{\left(\int_0^t \psi_m - \psi_n\right) \, \d L} >\frac{\epsilon}{3}\Bigg)=0.
\end{multline*}
By passing on to a suitable subsequence if necessary, we obtain that there exists a subsequence $(\int_0^{\cdot} \psi_{n_k} \, \d L)_{k \in \mathbb{N}}$ that converges uniformly almost surely, which guarantees that the limiting process has c\'adl\'ag paths.
\end{remark}

The following is the main result of this section identifying the largest space of $L$-integrable Hilbert-Schmidt operator-valued functions with the modular space $\mathcal{M}_{{\rm det},L}^{\rm HS}$.

\begin{theorem}\label{det_if_and_only_if_integrable}
The space $\mathcal{I}_{{\rm det},L}^{\rm HS}$ of deterministic functions integrable with respect to the cylindrical L\'evy process $L$ in $G$ coincides with the modular space $\mathcal{M}_{{\rm det},L}^{\rm HS}$.
\end{theorem}

The remainder of this section is devoted to proving the above theorem. As a first step, we prove a key Lemma, which shows that convergence of step functions in the modular topology is equivalent to convergence of the corresponding stochastic integrals in the following sense.

\begin{lemma}\label{le.cont_of_int_op}
Let $L$ be a cylindrical L\'evy process in $G$, and $(\psi_n)_{n\in{\mathbb N}}$ a sequence in $\mathcal{S}_{\rm det}^{\rm HS}$. Then the following are equivalent:
\begin{enumerate}
	\item[{\rm (a)}] $\displaystyle \lim_{n\to\infty}m_L(\psi_n)=0$;
	\item[{\rm (b)}] $\displaystyle \lim_{n\to\infty} \sup_{\gamma \in \mathcal{S}^{\rm 1,op}_{\rm det}}E\left[\norm{\int_0^T \gamma \psi_n \, dL}\wedge 1 \right]=0$ and $\displaystyle\lim_{n \rightarrow \infty}m''(\psi_n)=0$.
\end{enumerate}
\end{lemma}

The proof of the implication $(\rm a)\Rightarrow(\rm b)$ relies on two technical lemmata. The first of these gives a limit representation of the modular. Recall the notation $F(L)$ of the Radonified L\'evy process for an operator $F\in L_2(G,H)$ and a cylindrical L\'evy process $L$ in $G$ from Lemma \ref{radonif_of_cyl_levy}. 

\begin{lemma}\label{le.modular_as_limit}
 Let $(L(t):t \geq 0)$ be a cylindrical L\'evy process in $G$ with cylindrical characteristics $(a,Q,\lambda)$ and assume that $\psi \in \mathcal{S}_{\rm det}^{\rm HS}$ has the representation as in \eqref{eq.det-step-HS}. If $(\pi_k)_{k \in \mathbb{N}}$ is a nested normal sequence of partitions of $[0,T]$ containing the time points over which $\psi$ is defined,  then we have:
\begin{enumerate}
	\item[{\rm (a)}]
$\displaystyle 
     \lim_{k \rightarrow \infty}\sum_{i=1}^{n-1}\sum_{\substack{p_{j,k} \in \pi_k\\t_{i}<p_{j,k}\leq t_{i+1}}}E\left[\theta \big(\Phi_i(L)(p_{j,k})-\Phi_i(L)(p_{j-1,k})\big)\right]=\int_0^T b_{\psi(t)}^\theta \;{\rm d}t; 
$
\item[{\rm (b)}] 
$\displaystyle
     \lim_{k \rightarrow \infty}\sum_{i=1}^{n-1}\sum_{\substack{p_{j,k} \in \pi_k\\t_{i}<p_{j,k}\leq t_{i+1}}}E\left[\norm{\theta\big(\Phi_i(L)(p_{j,k})-\Phi_i(L)(p_{j-1,k})\big)}^2\right]$ \\
     $\displaystyle\hspace*{1em} = \int_0^T \int_H \left(\norm{h}^2 \wedge 1\right) \left(\lambda \circ \psi(t)^{-1}\right)({\rm d}h){\rm d}t + \int_0^T {\rm Tr}(\psi(t)Q\psi(t)^*)\,{\rm d}t.$
\end{enumerate}
\end{lemma}

\begin{proof}
The proof is a direct application of Lemma \ref{radonif_of_cyl_levy} and the limit characterisation of L\'evy characteristics in Theorem \ref{th:limit_characteristics}.
\end{proof}

Another ingredient of the proof of Lemma \ref{le.cont_of_int_op} is the following general result from \cite{kwapien_woyczynski_1992}.

\begin{lemma}\label{le.independent_sum_control}
For all $\epsilon>0$ there exists $\delta>0$ such that for any sequence $(X_n)_{n \in \{1,...,N\}}$ of independent $H$-valued random variables $X_n$ satisfying 
\[\norm{\sum_{n=1}^N E[\theta(X_n)]}<\delta \quad \text{and} \quad \sum_{n=1}^N E\left[\norm{\theta(X_n)}^2\right]< \delta,\]
it follows that 
\[E\left[\norm{\sum_{n=1}^N X_n} \wedge 1 \right]< \epsilon.\]
\end{lemma}

\begin{proof}
See \cite[Pr. 8.1.1/(ii)]{kwapien_woyczynski_1992}.
\end{proof}

\begin{proof}[Proof of $(\rm a)\Rightarrow(\rm b)$ in Lemma \ref{le.cont_of_int_op}.]
Let $\epsilon>0$ be fixed and choose $\delta>0$ so that the implication in Lemma \ref{le.independent_sum_control} holds. The hypothesis guarantees that there exists an $N \in \mathbb{N}$ such that for all $n\geq N$ we have $m(\psi_n)<\delta$.
Let $n_0\geq N$ and $\gamma_0 \in \mathcal{S}^{\rm 1,op}_{\rm det}$ be fixed. We assume the representation
\[\gamma_0 \psi_{n_0}(t)=O_{0,n_0}\Phi_{0,n_0}\mathbb{1}_{\{0\}}(t)+\sum_{i=0}^{N(n_0)-1}O_{i,n_0}\Phi_{i,n_0}\mathbb{1}_{(t_{i,n_0},t_{i+1,n_0}]}(t),\]
where $0=t_{0,n_0}<t_{1,n_0}<...<t_{N(n_0),n_0}=T$, $O_{i,n_0} \in \Bar{B}_{L(H)}$ and $\Phi_{i,n_0} \in L_2(G,H)$. Let $(\pi_k)_{k \in \mathbb{N}}$ be a nested normal sequence of partitions containing the points over which $\gamma_0 \psi_{n_0}$ is defined. Since by Lemma 
\ref{le.properties_of_k} and the very definition of $l_L$ we have $m_L(\gamma_0\psi_{n_0})\leq m_L(\psi_{n_0})<\delta$, Lemma \ref{le.modular_as_limit} guarantees that there exists a $K \in \mathbb{N}$ such that the partition $\pi_K$ satisfies
\begin{align*}
    \norm{\sum_{i=0}^{N(n_0)-1} \sum_{\substack{p_{j,K} \in \pi_K\\t_{i,n_0}<p_{j,K}\leq t_{i+1,n_0}}}E\left[\theta \big(O_{i,n_0}\Phi_{i,n_0}(L)(p_{j,K})-O_{i,n_0}\Phi_{i,n_0}(L)(p_{j-1,K})\big)\right]}<\delta
\end{align*}
and
\begin{align*}
    \sum_{i=0}^{N(n_0)-1} \sum_{\substack{p_{j,K} \in \pi_K\\t_{i,n_0}<p_{j,K}\leq t_{i+1,n_0}}}E\left[\norm{\theta\big(O_{i,n_0}\Phi_{i,n_0}(L)(p_{j,K})-O_{i,n_0}\Phi_{i,n_0}(L)(p_{j-1,K})\big)}^2\right]<\delta.
\end{align*}
Since $\pi_K$ contains the time points over which  $\gamma_0 \psi_{n_0}$ is defined,  Lemma \ref{le.independent_sum_control} implies 
\begin{align*}
    E\left[\norm{\int_0^T \gamma_0 \psi_{n_0} \, dL}\wedge 1 \right]
    &=E\left[\norm{\sum_{i=0}^{N(n_0)-1}O_{i,n_0}\Phi_{i,n_0}\left(L(t_{i+1,n_0})-L(t_{i,n_0})\right)} \wedge 1\right]
<\epsilon.
\end{align*}
This concludes the proof of the implication.
\end{proof}

The proof of the reverse implication $(\rm b)\Rightarrow(\rm a)$ in Lemma \ref{le.cont_of_int_op} relies on two technical lemmata. In the next proof, we follow closely the argument in \cite[Le. 3.9]{nowak_2003}.
%



\begin{lemma}\label{le.supremum_equivalency}
If $\psi \in \mathcal{S}_{\rm det}^{\rm HS}$ then
\[\int_0^T \sup_{O \in \Bar{B}_{L(H)}} \norm{b_{O \psi(t)}^\theta}\, {\rm d}t = \sup_{\gamma \in \mathcal{S}_{\rm det}^{\rm 1, op}} \norm{\int_0^T b_{\gamma \psi(t)}^\theta\, {\rm d}t},\]
where $\mathcal{S}_{\rm det}^{\rm 1, op}$ was defined in Definition \ref{def:step_fn}/(2).
\end{lemma}
\begin{proof}
Fix an element $e\in H$ such that $\norm{e}=1$. If $h \in H$ is linearly independent of $e$, we define $A_h:={\rm Span}\{e,h\}$. For each $h \in H$ we define the mapping $f(h):H \rightarrow \Bar{B}_{L(H)}$ by
\begin{equation}
    f(h)(h') =
    \begin{cases*}
      R_{A_h}(h_{A_h}')+h_{A_h^\bot}' ,& if $h \in H \setminus {\rm Span}\{e\}$ \\
      {\rm sgn}(\lambda)h'       ,& if $h= \lambda e$,
    \end{cases*}
  \end{equation}
where $h_{A_h}'$ and $h_{A_h^\bot}'$ denote the projections of $h'$ onto the subspace $A_h$ and its orthogonal complement $A_h^\bot$, respectively, and $R_{A_h}$ denotes the rotation on the plane $A_h$ around the origin by the angle $\cos^{-1}(\langle h,e\rangle)/\norm{h}$, that is, by the angle which rotates the vector $h$ into $e$. We claim that $f$ satisfies the following properties:
\begin{enumerate}
    \item [{\rm (1)}] for each $h \in H$ the mapping $f(h):H \rightarrow H$ is a linear isometry;
    \item [{\rm (2)}] for each $h \in H$ and $F \in L_2(G,H)$ we have $b_{f(h)F}^\theta= f(h)b_F^\theta$;
    \item [{\rm (3)}] for each $h \in H$ it holds that $\langle e,f(h)(h) \rangle = \norm{f(h)(h)}$.
\end{enumerate}

\textit{Proof of (1):} We first assume that $h \in H \setminus {\rm Span}\{e\}$. Then, since for any $h' \in H$ we have that $h'=h_{A_h}'+h_{A_h^\bot}'$, and $h_{A_h}'$ is orthogonal to $h_{A_h^\bot}'$, it follows that
\begin{align}\label{eq.isometry_1}
    \norm{h'}^2=\norm{h_{A_h}'+h_{A_h^\bot}'}^2&=\left\langle h_{A_h}'+h_{A_h^\bot}',h_{A_h}'+h_{A_h^\bot}'\right \rangle  
    = \norm{h_{A_h}'}^2+ \norm{h_{A_h^\bot}'}^2.
\end{align}
Since rotations are isometries, we have $\norm{R_{A_h}(h_{A_h}')}=\norm{h_{A_h}'}$. Moreover, it follows from the definition of the rotation $R_{A_h}$ that $R_{A_h}(h_{A_h}')$ is orthogonal to $h_{A_h^\bot}'$. Using these observations, a similar argument as above yields for all $h \in H \setminus {\rm Span}\{e\}$ and $h' \in H$ that
\begin{align}\label{eq.isometry_2}
    \norm{f(h)(h')}^2=\norm{R_{A_h}(h_{A_h}')+h_{A_h^\bot}'}^2&=\norm{R_{A_h}(h_{A_h}')}^2+\norm{h_{A_h^\bot}'}^2=\norm{h_{A_h}'}^2+\norm{h_{A_h^\bot}'}^2.
\end{align}
Hence, if $h \in H \setminus {\rm Span}\{e\}$ then by Equations \eqref{eq.isometry_1} and \eqref{eq.isometry_2} we have for all $h' \in H$ that $\norm{f(h)(h')}=\norm{h'}$. If, on the other hand, we have $h=\lambda e$ for some $\lambda \in \mathbb{R}$, then it follows from the very definition of $f$ that $\norm{f(h)(h')}=\norm{{\rm sgn}(\lambda)h'}=\norm{h'}$, which finishes the proof that for all $h \in H$ the mapping $f(h):H\rightarrow H$ is an isometry. Linearity of $f(h)$ follows directly from the definition.

\textit{Proof of (2):} By Lemma \ref{le.alternative_form_of_l} and the fact that by Step ${\rm (1)}$, for each $h_0 \in H$ the mapping $f(h_0):H \rightarrow H$ is a linear isometry, we obtain for all $\Phi \in L_2(G,H)$ that
\begin{align*}
    b_{f(h_0) \Phi}^\theta&= f(h_0)b_\Phi^\theta+ \int_H \large(\theta(f(h_0)h)-f(h_0) \theta(h)\large)\, (\lambda \circ \Phi^{-1})({\rm d}h)\nonumber\\
    &= f(h_0)b_\Phi^\theta+ \int_H \large(f(h_0)\theta(h)-f(h_0)\large) \theta(h)\, (\lambda \circ \Phi^{-1})({\rm d}h)= f(h_0)b_\Phi^\theta.
\end{align*}

\textit{Proof of (3):} Assume first that $h \in H \setminus {\rm Span}\{e\}$. Then, we have that $h_{A_h}=h$ and $h_{A_h^\bot}=0$. By combining these observations with the fact that by its very definition, $R_h$ rotates the vector $h$ into $e$, we get
\begin{align}\label{eq.inner_product_1}
    \langle e, f(h)(h) \rangle= \langle e, R_h(h) \rangle = \langle e, \norm{h}e \rangle = \norm{h}\langle e, e \rangle = \norm{h}= \norm{R_h(h)}= \norm{f(h)(h)}.
\end{align}
If, on the other hand, we assume $h=\lambda e$ for some $\lambda \in \mathbb{R}$, then
\begin{align}\label{eq.inner_product_2}
    \langle e, f(h)(h) \rangle = \langle e, {\rm sgn}(\lambda)\lambda e \rangle = \vert\lambda \vert = \norm{{\rm sgn}(\lambda)\lambda e}= \norm{f(h)h}.
\end{align}
Equations \eqref{eq.inner_product_1} and \eqref{eq.inner_product_2} establish the property (3).

To finish the proof of this lemma, let $\epsilon>0$ be fixed. We define a set-valued function $g:L_2(G,H) \rightarrow 2^{\Bar{B}_{L(H)}}$ by
\[g(\Phi)=\left\{ O \in \Bar{B}_{L(H)}: \sup_{Q \in \Bar{B}_{L(H)}}\norm{b_{Q\Phi}^\theta}-\norm{b_{O\Phi}^\theta}<\frac{\epsilon}{T} \right\}.\]
In order to prove the existence of a measurable selector for $g$, it suffices to show according to the  Kuratowski-Ryll Nardzewski measurable selection theorem, see e.g. \cite{KRN65}, that for all open sets $S \subseteq \Bar{B}_{L(H)}$ we have
\begin{align}\label{eq.KRN}
	\left\{ \Phi \in L_2(G,H): g(\Phi) \cap S \neq \emptyset \right\} \in \Borel(L_2(G,H)).
\end{align}
We define the mapping $\kappa: L_2(G,H) \times \Bar{B}_{L(H)} \rightarrow \mathbb{R}$ by
\[\kappa(\Phi, O)= \sup_{Q \in \Bar{B}_{L(H)}}\norm{b_{Q\Phi}^\theta}-\norm{b_{O\Phi}^\theta}.\]
It follows from lower-semicontinuity of $l_L$, see Lemma \ref{le.properties_k_l}, that for each fixed $O \in \Bar{B}_{L(H)}$ the mapping $\Phi \mapsto \kappa(\Phi,O)$ is lower-semicontinuous. Moreover, by \cite[Th.3]{p_summing}, \cite[Le. 2.1]{BR} and Equation \eqref{eq.cont_first_char}, we have that for each fixed $F \in L_2(G,H)$ the mapping $O \mapsto \kappa(F,O)$ is continuous with the strong topology. By using these observations and noting that $S$ is an open set in the strong topology, we obtain that
\begin{align*}
    \left\{ \Phi \in L_2(G,H): g(\Phi) \cap S \neq \emptyset \right\} &= \bigcup_{O \in S}\left\{ \Phi \in L_2(G,H): \kappa(\Phi,O)<\frac{\epsilon}{T} \right\}\\
    &=\bigcup_{O \in S\cap \mathcal{D}^1}\left\{ \Phi \in L_2(G,H): \kappa(\Phi,O)<\frac{\epsilon}{T} \right\},
\end{align*}
where $\mathcal{D}^1$ denotes a countable dense subset of $\Bar{B}_{L(H)}$ with the strong topology. Since for each fixed $O \in \Bar{B}_{L(H)}$, the mapping $\Phi \mapsto \kappa(\Phi,O)$ is lower-semicontinuous and hence measurable, for each fixed $O \in S$ it holds that
\[\left\{ \Phi \in L_2(G,H): \kappa(\Phi,O)< \frac{\epsilon}{T} \right\}\in \Borel(L_2(G,H)).\]
The Kuratowski-Ryll Nardzewski measurable selection theorem implies that there exists a measurable selector function $i\colon L_2(G,H) \rightarrow \Bar{B}_{L(H)}$ satisfying for all $\Phi \in L_2(G,H)$ that $i(\Phi) \in g(\Phi)$.

Finally, we define a function  
\[\eta:[0,T]\rightarrow \Bar{B}_{L(H)}, \qquad \eta(t)=f\left(b_{i(\psi)\psi}^\theta\right)\big( i(\psi(t))\big),\]
where $\eta$ is measurable since it is the composition of measurable functions. Using properties ${\rm (1)}-{\rm (3)}$ of the mapping $f$, we obtain
\begin{align*}\label{eq.first_char_epsilon}
    \norm{\int_0^T b_{\eta \psi(t)}^\theta\, {\rm d}t} &\geq \int_0^T \langle b_{\eta \psi(t)}^\theta,e \rangle\, {\rm d}t\\
    &= \int_0^T \norm{b_{i(\psi(t))\psi(t)}^\theta}\, {\rm d}t\nonumber
     \geq \int_0^T \left(\sup_{O \in \Bar{B}_{L(H)}}\norm{b_{O\psi(t)}^\theta}-\frac{\epsilon}{T}\right)\, {\rm d}t. 
\end{align*}
Since $\epsilon>0$ is arbitrary, and by approximating $\eta$ using processes from $\mathcal{S}_{\rm det}^{\rm 1, op}$ we conclude
\[\sup_{\gamma \in \mathcal{S}_{\rm det}^{\rm 1, op}}\norm{\int_0^T b_{\gamma \psi}^\theta\, {\rm d}t}\geq\int_0^T \sup_{O \in \Bar{B}_{L(H)}}\norm{b_{O\psi(t)}^\theta}\, {\rm d}t.\]
As the reverse inequality follows directly from basic properties of the Bochner integral, the proof is complete.
\end{proof}

\begin{lemma}\label{le.conv_of_first_char}
If a sequence $(\psi_n)_{n \in \mathbb{N}} \subseteq \mathcal{S}_{\rm det}^{\rm HS}$ satisfies
\[\lim_{n\to\infty} \sup_{\gamma \in \mathcal{S}^{\rm 1,op}_{\rm det}}E\left[\norm{\int_0^T \gamma \psi_n \, {\rm d}L}\wedge 1 \right]=0,\]
then
\[\lim_{n \rightarrow \infty}\sup_{\gamma \in \mathcal{S}_{\rm det}^{\rm 1, op}} \norm{\int_0^T b_{\gamma \psi_n(t)}^\theta\, {\rm d}t}=0.\]
\end{lemma}

\begin{proof}
Assume, aiming for a contradiction, that this is not the case. Then, by passing on to a suitable subsequence if necessary, there exists an $\epsilon>0$ and a sequence $(\gamma_n)_{n \in \mathbb{N}} \subseteq \mathcal{S}_{\rm det}^{\rm 1, op}$ satisfying for all $n \in \mathbb{N}$ that
\begin{equation}\label{eq.first_char_contradiction}
    \norm{\int_0^T b_{\gamma_n \psi_n(t)}^\theta\,{\rm d}t}>\epsilon.
\end{equation}
On the other hand, the hypothesis 
implies that the sequence $(I(\gamma_n \psi_n))_{n \in \mathbb{N}}$ of infinitely divisible random variables
$I(\gamma_n\psi_n)$, whose first characteristics equals $\int_0^T b_{\gamma_n\psi_n(t)}^{\theta}\,{\rm d}t$, converges to $0$ in probability. It follows from Equation \eqref{eq.cont_first_char} that
\[\lim_{n \rightarrow \infty}\int_0^T b_{\gamma_n\psi_n(t)}^{\theta}\,dt=0,\]
which contradicts Equation (\ref{eq.first_char_contradiction}). Hence, the result follows.
\end{proof}

Having provided all the necessary preliminary results, we now proceed to proving the reverse implication in Lemma \ref{le.cont_of_int_op}.

\begin{proof}[Proof of $(\rm b)\Rightarrow(\rm a)$ in Lemma \ref{le.cont_of_int_op}.]
Since for each $n \in \mathbb{N}$, $\psi_n$ has a representation of the form
\[\psi_n(t)=\Phi_0^n\mathbb{1}_{\{0\}}(t)+\sum_{i=1}^{N(n)-1} \Phi_i^n \mathbb{1}_{(t_i^n,t_{i+1}^n]}(t),\]
where $0= t_1^n<...<t_{N(n)}^n = T$, and $\Phi_i^n \in L_2(G,H)$ for each $i \in \{0,...,N(n)-1\}$, the integral $I(\psi_n)$ takes the form
\[I(\psi_n)=\sum_{i=1}^{N(n)-1}\Phi_i^n \left(L(t_{i+1}^n)-L(t_i^n)\right). \] 
Since the integral $I(\psi_n)$ is the sum of independent infinitely divisible random variables, $I(\psi_n)$ is also infinitely divisible and has characteristics
\[\left(\sum_{i=1}^{N(n)-1}(t_{i+1}^n-t_i^n)b_{\Phi_i^n}^\theta,\sum_{i=1}^{N(n)-1}(t_{i+1}^n-t_i^n)\Phi_i^n Q (\Phi_i^n)^*,\sum_{i=1}^{N(n)-1} (t_{i+1}^n-t_i^n) \left(\lambda \circ (\Phi_i^n)^{-1}\right)\right).\]
Since the hypothesis  implies that $\lim_{n \rightarrow \infty}I(\psi_n)=0$ in $L_P^0(\Omega,H)$, we conclude from Lemma \ref{le.measure_theoretic_lemma} and \cite[Le. 3.6]{BR} that
\begin{align}\label{eq.m'_estimate}
    &\lim_{n \rightarrow \infty}\int_0^T k_L(\psi(t))\,dt \nonumber\\
    &=\lim_{n \rightarrow \infty}\int_H \left(\norm{h}^2 \wedge 1\right)\left(\left(\lambda \otimes {\rm Leb} \right)\circ \kappa_{\psi_n}^{-1} \right) ({\rm d}h)+ {\rm Tr}\left(\int_0^T\left(\psi_n(t)Q\psi_n(t)^*\right)\; {\rm d}t\right)=0.
\end{align}
Furthermore, the hypothesis implies  by Lemmata \ref{le.supremum_equivalency} and \ref{le.conv_of_first_char} that
\begin{align}\label{eq.m''_estimate}
    \lim_{n \rightarrow \infty}\int_0^T l_L(\psi_n(t))dt&=\lim_{n \rightarrow \infty} \int_0^T \sup_{O \in \Bar{B}_{L(H)}} \norm{b_{O \psi_n(t)}^\theta}\, {\rm d}t \nonumber\\
    &= \lim_{n \rightarrow \infty}\sup_{\gamma \in \mathcal{S}_{\rm det}^{\rm 1, op}} \norm{\int_0^T b_{\gamma \psi_n(t)}^\theta\, {\rm d}t}=0.
\end{align}
Equations (\ref{eq.m'_estimate}) and (\ref{eq.m''_estimate}) together imply that $\lim_{n \rightarrow \infty}m_L'(\psi_n)=0$. Since by assumption, we have $\lim_{n \rightarrow \infty}m''(\psi_n)=0$, we obtain that $\lim_{n \rightarrow \infty}m_L(\psi_n)=0$.
\end{proof}

We are now ready to present the proof of the main result of this section characterising the largest space of deterministic Hilbert-Schmidt operator-valued functions which are integrable with respect to a cylindrical L\'evy process $L$ in Hilbert space.

\begin{proof}[Proof of Theorem \ref{det_if_and_only_if_integrable}]
If $\psi \in \mathcal{I}_{{\rm det},L}^{\rm HS}$ then by the very definition of integrability, see Definition \ref{det_integrability}, there exists a sequence $(\psi_n)_{n \in \mathbb{N}} \subseteqq\mathcal{S}_{\rm det}^{\rm HS}$ such that $\psi_n \rightarrow \psi$ Lebesgue a.e.\ and  $\sup_{\gamma \in \mathcal{S}_{\rm det}^{1,{\rm op}}}E[\norm{I(\gamma(\psi_n-\psi_m))}\wedge 1] \rightarrow 0$. By Lemma \ref{le.cont_of_int_op}, this implies that $m_L(\psi_n-\psi_m) \rightarrow 0$. Completeness of the modular space $\mathcal{M}_{{\rm det},L}^{\rm HS}$, see Lemma \ref{le.completeness_modular_top}, and the fact that $\psi_n \rightarrow \psi$ Lebesgue a.e.\ allows us to conclude that $\psi \in \mathcal{M}_{{\rm det},L}^{\rm HS}$.

Conversely, if  $\psi \in \mathcal{M}_{{\rm det},L}^{\rm HS}$, then Lemma \ref{le.density_of_simple_fn} implies that there exists a sequence $(\psi_n)_{n \in \mathbb{N}}$ of elements in $\mathcal{S}_{\rm det}^{\rm HS}$ such that $\psi_n \rightarrow \psi$ Lebesgue a.e.\ and $m_L(\psi_n-\psi)\rightarrow 0$. It follows that $m_L(\psi_n-\psi_m)\rightarrow 0$, which implies  $\sup_{\gamma \in \mathcal{S}_{\rm det}^{\rm 1,op}}E[\norm{I(\gamma(\psi_n-\psi_m))}\wedge 1] \rightarrow 0$ by Lemma \ref{le.cont_of_int_op} and establishes that $\psi \in \mathcal{I}_{{\rm det},L}^{\rm HS}$. 
\end{proof}

\section{Stochastic integrals with predictable integrands}\label{se.stochastic-integrals}

For the remainder of this section, we fix a filtered probability space $(\Omega, \mathcal{F}, (\mathcal{F}_t)_{t \geq 0}, P)$. As in the case of deterministic integrands, we begin by introducing two classes of functions on which our definition of the stochastic integral depends on.
\begin{definition} \hfill 
	\begin{enumerate}
		\item[{\rm (1)}] An $L_2(G,H)$-valued predictable step process $\Psi \colon \Omega \times [0,T]\rightarrow L_2(G,H)$ is of the form 
		\begin{align} \label{eq.step-HS}
		  \Psi(\omega, t)=\left(\sum_{k=1}^{N(0)}F_{0,k}\mathbb{1}_{A_{0,k}}(\omega)\right)\mathbb{1}_{\{0\}}(t)+\sum_{i=1}^{n-1} \left(\sum_{k=1}^{N(i)}F_{i,k}\mathbb{1}_{A_{i,k}}(\omega)\right) \mathbb{1}_{ (t_i,t_{i+1}]}(t),
		\end{align}
		where $0=t_1<\cdots < t_{n}=T$, $A_{0,k} \in \mathcal{F}_0$ and $F_{0,k} \in L_2(G,H)$ for all $k=1,...,N(0)$, $A_{i,k} \in \mathcal{F}_{t_i}$ and $F_{i,k} \in L_2(G,H)$ for all $i=1,...,n-1$ and $k=1,...,N(i)$. The space of all $L_2(G,H)$-valued predictable step processes is denoted by $\mathcal{S}_{\rm prd}^{\rm HS}:=\mathcal{S}_{\rm prd}^{\rm HS}(G,H)$.
		\item[{\rm (2)}] An $L(H)$-valued predictable step process $\Gamma \colon \Omega \times [0,T]\rightarrow L(H,H)$ is of the form
		\begin{align}
		  \Gamma (\omega, t)=\left(\sum_{k=1}^{N(0)}O_{0,k}\mathbb{1}_{A_{0,k}}(\omega)\right)\mathbb{1}_{\{0\}}(t)+\sum_{i=1}^{n-1} \left(\sum_{k=1}^{N(i)}O_{i,k}\mathbb{1}_{A_{i,k}}(\omega)\right) \mathbb{1}_{ (t_i,t_{i+1}]}(t),
		\end{align}
			where $0=t_1<\cdots < t_{n}=T$, $A_{0,k} \in \mathcal{F}_0$ and $O_{0,k}\in L(H)$ for all $k=1,...,N(0)$, $A_{i,k} \in \mathcal{F}_{t_i}$ and $O_{i,k} \in L(H)$ for all $i=1,...,n-1$ and $k=1,...,N(i)$. The space of all $L(H)$-valued predictable step processes with
			\begin{align*}
			    \sup_{(\omega,t)\in \Omega \times [0,T]}\norm{\Gamma(\omega,t)}_{H \rightarrow H}\leq 1
			\end{align*}
			is denoted by  $\mathcal{S}_{{\rm prd}}^{1, {\rm op}}:=\mathcal{S}_{{\rm prd}}^{1, {\rm op}}(H,H)$.
	\end{enumerate}
\end{definition}

Let $\Psi \in \mathcal{S}_{\rm prd}^{\rm HS}$ be of the form (\ref{eq.step-HS}) and $L$ a cylindrical L\'evy process in $G$. As explained after Definition \ref{def:step_fn},  there exists an $H$-valued random variable $\Phi_{i,k}(L(t_{i+1})-L(t_i))\colon \Omega\to H$ for each $i=1,...,n-1$ and $k=1,...,N(i)$, satisfying \[ \big(L(t_{i+1})-L(t_i)\big)(\Phi_{i,k}^*h)=\langle \Phi_{i,k}(L(t_{i+1})-L(t_i)),h \rangle \quad \text{$P$-a.s.\ for all}\; h \in H.\]
In this case,  the stochastic integral of $\Psi$ is defined by
\[I(\Psi):=\int_0^T \Psi(t) \, {\rm d}L(t) :=\sum_{i=1}^{n-1}\sum_{k=1}^{N(i)} \mathbb{1}_{A_{i,k}}\Phi_{i,k} (L(t_{i+1})-L(t_i) ).\]
Thus, the integral $I(\Psi):\Omega\rightarrow H$ is a genuine $H$-valued random variable. In contrast to the deterministic case in Section \ref{se.deterministic}, the integral $I(\Psi)$ is not necessarily infinitely divisible. 

For the purposes of this section, it is convenient to introduce the measure space $\big(\Omega \times [0,T],\mathcal{P}, P_T\big)$, where $\mathcal{P}$ denotes the predictable $\sigma$-algebra and the measure $P_T$ is defined by $P_T:=P \otimes {\rm Leb} \vert_{[0,T]}$.

\begin{definition} \label{pred_integrability}
We say that a predictable process $\Psi$ is $L$-integrable if there exists a sequence $(\Psi_n)_{n \in \mathbb{N}}$ of processes in $\mathcal{S}_{\rm prd}^{\rm HS}$ such that
\begin{enumerate}[(1)]
    \item[{\rm (1)}]  $(\Psi_n)_{n \in \mathbb{N}}$ converges $P_T$-a.e. to $\Psi$;
    \item[{\rm (2)}] $\displaystyle \lim_{m,n \rightarrow \infty}\sup_{\Gamma \in \mathcal{S}_{{\rm prd}}^{1, {\rm op}}}E\Bigg[\norm{\int_0^T \Gamma(\Psi_m-\Psi_n) \;{\rm d}L}\wedge1 \Bigg]=0.$
\end{enumerate}
In this case,  the stochastic integral of $\Psi$ is defined by
\[I(\Psi):=\int_0^T \Psi \;{\rm d}L= \lim_{n\rightarrow \infty} \int_0^T \Psi_n \;{\rm d}L \quad \text{in}\;L_P^0(\Omega,H).\]
The class of all $L$-integrable $L_2(G,H)$-valued predictable processes will be denoted by $\mathcal{I}_{{\rm prd},L}^{\rm HS}:=\mathcal{I}_{{\rm prd},L}^{\rm HS}(G,H)$. As usual, for $t \in [0,T]$, we define $\int_0^t \Psi \;{\rm d}L:=\int_0^T \mathbb{1}_{\Omega \times (0,t]}\Psi \;{\rm d}L$.
\end{definition}

\begin{remark}\label{re.pred_ucp}
An extension of \cite[Le. 2.3]{kw_via_dec} to $H$-valued random variables shows that Condition ${\rm (2)}$ of Definition \ref{pred_integrability} implies
\[\lim_{m,n \rightarrow \infty}\sup_{\Gamma \in \mathcal{S}_{{\rm prd}}^{1, {\rm op}}}E\Bigg[\sup_{t \in [0,T]}\norm{\int_0^t \Gamma(\Psi_m-\Psi_n) \;{\rm d}L}\wedge1 \Bigg]=0.\]
Hence, the notion of convergence introduced in Definition \ref{pred_integrability} coincides with convergence in the \'Emery topology. Since convergence in the \'Emery topology implies convergence in the ucp topology, this immediately gives that for each $\Psi \in \mathcal{I}_{{\rm prd},L}^{\rm HS}$ the process $\left(\int_0^t \Psi\, {\rm d}L\right)_{t\in [0,T]}$ has c\'adl\'ag paths. For details, see the end of Remark \ref{cadlag_remark}.
\end{remark}

Let $L$ be a cylindrical L\'evy process in $G$ and denote by $\rho_L$ the metric on the modular space $\mathcal{M}_{{\rm det},L}^{\rm HS}$ making this space Polish; see Proposition \ref{pr.m_polish}. Let $L_P^0(\Omega, \mathcal{M}_{{\rm det},L}^{\rm HS})$ denote the collection of all Borel-measurable mappings $\Psi:\Omega \rightarrow \mathcal{M}_{{\rm det},L}^{\rm HS}$ which we endow with the translation invariant metric
\[\normm{\Psi_1-\Psi_2}_L:=E\left[\rho_L(\Psi_1-\Psi_2)\wedge 1\right] \quad \text{for }\Psi_1,\Psi_2 \in L_P^0(\Omega, \mathcal{M}_{{\rm det},L}^{\rm HS}),\]
which guarantees its completeness. 

\begin{lemma} \label{pred_density}
Let $\Psi$ be a predictable stochastic process in $L_P^0(\Omega, \mathcal{M}_{{\rm det},L}^{\rm HS})$. Then there exists a sequence $(\Psi_k)_{k \in \mathbb{N}}$ of elements of $\mathcal{S}_{\rm prd}^{\rm HS}$
converging to $\Psi$ both in the metric $\normm{\cdot}_L$ and $P_T$-a.e. 
\end{lemma}

\begin{proof}
If $\Psi$ is bounded, then $\Psi \in L_{P_T}^{\infty}\big(\Omega \times [0,T],L_2(G,H)\big)$. Since the algebra of sets
\[\mathcal{A'}=\big\{(s,t]\times B:s<t,B \in \mathcal{F}_s\big\} \cup \big\{\{0\}\times B:B \in \mathcal{F}_0\big\}\]
generates $\mathcal{P}$, we conclude from \cite[Le.\ 1.2.19]{hytonen_neerven_2016} and \cite[Re.\ 1.2.20]{hytonen_neerven_2016} that there exists a sequence $(\Psi_k)_{k \in \mathbb{N}}$ of uniformly bounded processes in $\mathcal{S}_{\rm prd}^{\rm HS}$ such that $\Psi_k \rightarrow \Psi$ $P_T$-a.e. Thus, there exists a set $N \in \mathcal{P}$ such that $P_T(N)=0$ and $\big(\Psi_k(\omega,t)-\Psi(\omega,t)\big)\rightarrow 0$ for all $(\omega,t) \in N^c$. Fubini's theorem implies that
\[0=P_T(N)=P \otimes {\rm Leb}\vert_{[0,T]}(N)=\int_{\Omega}{\rm Leb}\vert_{[0,T]}(N_{\omega})\,P({\rm d}\omega),\]
where for each fixed $\omega \in \Omega$ we define
\[N_\omega:=\Big\{t \in [0,T]\colon  \big(\Psi_k(\omega,t)-\Psi(\omega,t)\big)_{m \in \mathbb{N}} \; \text{does not converge to 0}\Big\}.\]
 The above implies that ${\rm Leb}\vert_{[0,T]}(N_\omega)=0$ for almost all $\omega \in \Omega$, that is, there exists an $\Omega_0\subseteq \Omega$ with $P(\Omega_0)=1$ such that for all $\omega \in \Omega_0$ we have 
 \[{\rm Leb}\vert_{[0,T]}\Big(t \in [0,T]\colon  \big(\Psi_k(\omega,t)-\Psi(\omega,t)\big)_{m \in \mathbb{N}} \; \text{does not converge to 0}\Big)=0.\]
Because  $(\Psi_k)_{k \in \mathbb{N}}$ is uniformly bounded and $\Psi$ is bounded, we can conclude from   Lebesgue's dominated convergence theorem  that $\lim_{n \rightarrow \infty}m_L\left(\Psi_k(\omega,\cdot)-\Psi(\omega,\cdot)\right)=0$ for each $\omega \in \Omega_0$. Since $m_L$ and $\rho_L$ generate the same topology on $\mathcal{M}_{{\rm det},L}^{\rm HS}$, we also have $\rho_L(\Psi_k(\omega,\cdot)-\Psi(\omega,\cdot))\rightarrow 0$ as $k \rightarrow \infty$ for each $\omega \in \Omega_0$.  Another application of Lebesgue's dominated convergence theorem yields
\begin{align*}
    \lim_{k \rightarrow \infty}\normm{\Psi_k-\Psi}_L= \lim_{k \rightarrow \infty} \int_{\Omega}\Big(\rho_L(\Psi_k(\omega,\cdot)-\Psi(\omega,\cdot))\wedge 1\Big)\,{\rm d}P=0,
\end{align*}
which shows the claim  if $\Psi$ is bounded. In the case of a general $\Psi$, we define
\[\Psi_n:\Omega\times [0,T] \rightarrow L_2(G,H),\qquad
    \Psi_n(\omega,t)= 
\begin{cases}
    \Psi(\omega,t) & \text{if } \norm{\Psi(\omega,t)}_{L_2(G,H)}\leq n , \\
    0              & \text{otherwise.}
\end{cases}
\]
Clearly, $\lim_{n \rightarrow \infty}\normm{\Psi-\Psi_n}_L=0$. The first part of the proof shows that 
for each $n \in \mathbb{N}$ there exists a sequence $(\Psi_{n,k})_{k \in \mathbb{N}} \subseteq \mathcal{S}_{\rm prd}^{\rm HS}$ converging to $\Psi_n$ as $k\to\infty$  in $\normm{\cdot}_L$ and $P_T$-a.e. For each $n \in \mathbb{N}$  choose $k_n\in {\mathbb N}$ such that $\normm{(\Psi_n-\Psi_{n,k_n})}_L < \frac{1}{n}$.
It follows that 
\[\lim_{n \rightarrow \infty}\normm{(\Psi-\Psi_{n,k_n})}_L\leq \lim_{n \rightarrow \infty}\bigg(\normm{(\Psi-\Psi_n)}_L+\normm{(\Psi_n-\Psi_{n,k_n})}_L\bigg)=0,\]
which completes the proof, since by passing on to a suitable subsequence, we also have convergence $P_T$-a.e.
\end{proof}

\section{Construction of the decoupled tangent sequence}\label{se.deoupled-tangent}

The technique of constructing decoupled tangent sequences is a powerful tool to obtain strong results on a sequence of possibly dependent random variables. In this section, we briefly recall the fundamental definition, see e.g.\ Kwapie{\'n} and Woyczy{\'n}ski \cite{kwapien_woyczynski_1992} or  de la Pe\~{n}a and Gin\'{e} \cite{Pena_1999}, and construct the decoupled tangent sequence in our setting which will enable us to identify the largest space of predictable integrands in the next section. 

\begin{remark} \label{product_expectation_remark}
We repeatedly use the fact in the following that given a random variable $X$ on $(\Omega,\mathcal{F},P)$ and another probability space $(\Omega',\mathcal{F'},P')$, the random variable $X$ can always be considered as a random variable on the product space $(\Omega \times \Omega',\mathcal{F} \otimes \mathcal{F}',P \otimes P')$ by defining
\[X(\omega,\omega')=X(\omega)\quad \text{for all}\;(\omega,\omega')\in \Omega \times \Omega'.\]
In this case, if $X$ is real-valued and $P$-integrable we have  $E_P[X]=E_{P \otimes P'}[X]$.
\end{remark}

\noindent In the next definition, we follow closely Chapter 4.3 of \cite{kwapien_woyczynski_1992}.

\begin{definition}\label{dec_tang_seq_def}
Let $\big(\Omega, \mathcal{F},P,(\mathcal{F}_n)_{n\in{\mathbb N}}\big)$ be a filtered probability space and $(X_n)_{n \in \mathbb{N}}$ an $(\mathcal{F}_n)$-adapted sequence of $H$-valued random variables. If $\big(\Omega', \mathcal{F'},P',(\mathcal{F'}_n)_{n\in{\mathbb N}}\big)$ is another filtered probability space, then a sequence $(Y_n)_{n \in \mathbb{N}}$ of $H$-valued random variables defined on $\big(\Omega \times \Omega^\prime, \mathcal{F}\otimes \mathcal{F'}, P \otimes P', (\mathcal{F}_n\otimes \mathcal{F'}_n)_{n\in{\mathbb N}}\big)$ is said to be a decoupled tangent sequence to $(X_n)_{n \in \mathbb{N}}$ if
\begin{enumerate}[(1)]
    \item[{\rm (1)}]  for each $\omega \in \Omega$, we have that $(Y_n(\omega,\cdot))_{n \in \mathbb{N}}$ is a sequence of independent random variables on $(\Omega',\mathcal{F}',P')$;
    \item[{\rm (2)}]  the sequences $(X_n)_{n \in \mathbb{N}}$ and $(Y_n)_{n \in \mathbb{N}}$ satisfy for each $n\in {\mathbb N}$ that
\[\mathcal{L}(X_n \vert \mathcal{F}_{n-1}\otimes \mathcal{F'}_{n-1})=\mathcal{L}(Y_n \vert \mathcal{F}_{n-1}\otimes \mathcal{F'}_{n-1}) \quad P \otimes P'-\text{a.s.}\]
\end{enumerate}
\end{definition}

\begin{remark}
    The importance of decoupled tangent sequences within the framework of stochastic integration lies in the existence of a collection of inequalities, frequently called decoupling inequalities, which relate convergence of an adapted sequence of random variables to convergence of their decoupled tangent sequence. In particular, by \cite[Pr.\ 5.7.1.(ii)]{kwapien_woyczynski_1992}, there exists a constant $c_1>0$ such that for all finite adapted sequences $(X_n)_{n=1,...,N}$ of $H$-valued random variables with corresponding decoupled tangent sequence $(Y_n)_{n=1,...,N}$ it holds that
    \[E_P\left[\norm{\sum_{n=1}^N X_n}\wedge 1\right]\leq c_1 E_{P\otimes P}\left[\norm{\sum_{n=1}^N Y_n}\wedge 1\right].\]
    Moreover, by \cite[Pr.\ 5.7.2]{kwapien_woyczynski_1992}, there exists $c_2>0$ such that the following "recoupling" inequality also holds
    \[E_{P\otimes P}\left[\norm{\sum_{n=1}^N Y_n}\wedge 1\right]\leq c_2 \sup_{\epsilon_n \in \{\pm 1\}}E_P\left[\norm{\sum_{n=1}^N \epsilon_n X_n}\wedge 1\right].\]
\end{remark}

The main tool for establishing the stochastic integral in the next section is a cylindrical L\'evy process $\widetilde{L}$ on an enlarged probability space, whose Radonified increments are decoupled to the Radonified increments of the original cylindrical L\'evy process. This cylindrical L\'evy process $\widetilde{L}$ is explicitly constructed in the following result.

\begin{proposition} \label{dec_tan_seq}
Let $L$ be a cylindrical L\'evy process in $G$, $0 = t_0 \leq ... \leq t_N=T$ be a partition of $[0,T]$	
and for each $n=1,...,N$ we define $\Theta_n:= \sum_{k=1}^{M(n)} \Phi_{n,k} \mathbb{1}_{A_{n,k}}$, where $\Phi_{n,k} \in L_2(G,H)$,  $A_{n,k} \in \mathcal{F}_{t_{n-1}}$ for all $k=1,...,M(n)$.
By defining cylindrical random variables 
\[ \widetilde{L}(t)\colon G \rightarrow L^0_{P \otimes P}(\Omega \times \Omega; \mathbb{R}), \qquad 
\Big(\widetilde{L}(t)g\Big)(\omega,\omega')=\Big(L(t)g\Big)(\omega'),\]
it follows that $(\widetilde{L}(t):t\geq 0)$  is a cylindrical L\'evy process on $G$  and the sequence of its Radonified increments
\[\left(\Theta_n \big(\widetilde{L}(t_n)-\widetilde{L}(t_{n-1})\big)\right)_{n \in \{1,...,N\} }\]
defined on $\big(\Omega \times \Omega,\mathcal{F} \otimes \mathcal{F},P \otimes P,(\mathcal{F}_{t_n} \otimes \mathcal{F}_{t_n})_{n \in \{0,...,N\}}\big)$ is a decoupled tangent sequence to the sequence of Radonified increments
\[\Big(\Theta_n \big(L(t_n)-L(t_{n-1})\big)\Big)_{n \in \{1,...,N\}}\]
defined on $\big(\Omega,\mathcal{F},P,(\mathcal{F}_{t_n})_{n \in \{0,...,N\}}\big).$
\end{proposition}

\begin{proof}
In order to make it easier to follow this proof, we define $\Omega'=\Omega$, $\mathcal{F'}=\mathcal{F}$, $P'=P$ and $\mathcal{F'}_{t_n}=\mathcal{F}_{t_n}$ for all $n \in \{0,...,N\}$ and instead of denoting the filtered product space by
\[\Big(\Omega \times \Omega,\mathcal{F} \otimes \mathcal{F},P \otimes P,(\mathcal{F}_{t_n} \otimes \mathcal{F}_{t_n})_{n \in \{0,...,N\}}\Big),\]
we write
\[\Big(\Omega \times \Omega',\mathcal{F} \otimes \mathcal{F'},P \otimes P',(\mathcal{F}_{t_n} \otimes \mathcal{F'}_{t_n})_{n \in \{0,...,N\}}\Big).\]
The fact that for each $t\geq 0$ the mapping $\widetilde{L}(t)\colon G \rightarrow L_{P \otimes P'}^0(\Omega \times \Omega',\mathbb{R})$ is continuous follows directly from the definition of $\widetilde{L}$ and Remark \ref{product_expectation_remark}. Thus $\widetilde{L}$ is a cylindrical stochastic process. To prove that it is in fact a cylindrical L\'evy process, let us fix $n \in \mathbb{N}$ and $g_1,\ldots ,g_n \in G$ and consider the $n$-dimensional processes
$Y$ and $Z$ defined by $Y(t)=(\widetilde{L}(t)g_1,\ldots,\widetilde{L}(t)g_n)$ and $Z(t)=(L(t)g_1,\ldots, L(t)g_n)$.
It is enough to show that for any $m \in \mathbb{N}$ and times $0\leq t_0 <\cdots <t_m\leq T$ the random variables $Y(t_m)-Y(t_{m-1}), \ldots, Y(t_1)-Y(t_0)$ and $Z(t_m)-Z(t_{m-1}), \dots ,Z(t_1)-Z(t_0)$
have the same distribution. Here we only prove that for any $0\leq s<t \leq T$ the random variables $Y(t)-Y(s)$ and $Z(t)-Z(s)$ have the same distribution. The general case follows analogously. To see this, let $A \in \Borel(\mathbb{R}^n)$ be arbitrary. The very definition of $\widetilde{L}$ shows 
\begin{align*}
    &(P \otimes P')\left(Y(t)-Y(s)\in A\right)\\
    &\qquad =(P \otimes P')\left((\widetilde{L}(t)g_1-\widetilde{L}(s)g_1,...,\widetilde{L}(t)g_n-\widetilde{L}(s)g_n)\in A\right)\\
    &\qquad =(P \otimes P')\left(\Omega \times \left\{(L(t)g_1-L(s)g_1,...,L(t)g_n-L(s)g_n)\in A\right\}\right)\\
    &\qquad=P'\left( (L(t)g_1-L(s)g_1,...,L(t)g_n-L(s)g_n)\in A \right)\\
    &\qquad =P\left(Z(t)-Z(s)\in A\right).
\end{align*}
To show that the Radonified increments of $\tilde{L}$ satisfy Condition (1) of Definition \ref{dec_tang_seq_def},  fix some $\omega \in \Omega$. Then $\Theta_n(\omega)$ is a (deterministic) Hilbert-Schmidt operator and $(\widetilde{L}(t)(\omega,\cdot):t \geq 0)$ is a cylindrical Lévy process in $G$. Thus, for a fixed $\omega \in \Omega$ and $n\in \{1,...,N\}$, the Radonified increment $\Theta_n(\omega) (\widetilde{L}(t_n)(\omega,\cdot)-\widetilde{L}(t_{n-1})(\omega,\cdot))$ is an $\mathcal{F}'_{t_n}$-measurable $H$-valued random variable on $(\Omega',\mathcal{F}',P')$ independent of $\mathcal{F}'_{t_{n-1}}$. It follows for each $\omega \in \Omega$ that
\[\left(\Theta_n(\omega) (\widetilde{L}(t_n)(\omega,\cdot)-\widetilde{L}(t_{n-1})(\omega,\cdot))\right)_{n \in \{1,...,N\}}\]
is a sequence of independent random variables. 

For establishing Condition (2) of Definition \ref{dec_tang_seq_def}, we define for each $n \in \{1,...,N\}$
the $H$-valued random variables 
\begin{align*}
&X_n:=\Theta_n\big(L(t_n)-L(t_{n-1})\big)= \sum_{k=1}^{M(n)} \mathbb{1}_{A_{n,k}} \Phi_{n,k} \big(L(t_n)-L(t_{n-1})\big),\\
&Y_n:=\Theta_n\big(\widetilde{L}(t_n)-\widetilde{L}(t_{n-1})\big)
=\sum_{k=1}^{M(n)} \mathbb{1}_{A_{n,k}} \Phi_{n,k} \big(\widetilde{L}(t_n)-\widetilde{L}(t_{n-1})\big),
\end{align*}
where $F_{n,k} \big(L(t_n)-L(t_{n-1})\big)$ and $ F_{n,k} \big(\widetilde{L}(t_n)-\widetilde{L}(t_{n-1})\big)$ refer to the Radonified increments, and by taking another representation of $\Theta_n$ if necessary, we may assume that for each $n \in \mathbb{N}$ the representation of $\Theta_n$ satisfies $A_{n,k}\cap A_{n,l}=\emptyset$ for $k\neq l$ and $\cup_{k=1}^{M(n)}A_{n,k}=\Omega$. Choose regular versions of the conditional distributions
\begin{align*}
    &(P \otimes P')_{X_n}\colon \Borel(H) \times (\Omega \times \Omega') \rightarrow [0,1],\\
    &\qquad\qquad (P \otimes P')_{X_n}\big(B,(\omega,\omega')\big)=(P \otimes P')(X_n \in B \vert \mathcal{F}_{t_{n-1}} \otimes \mathcal{F'}_{t_{n-1}})(\omega,\omega'),\\
    & (P \otimes P')_{Y_n}\colon \Borel(H) \times (\Omega \times \Omega') \rightarrow [0,1], \\ 
     &\qquad\qquad  (P \otimes P')_{Y_n}\big(B,(\omega,\omega')\big)=(P \otimes P')(Y_n \in B \vert \mathcal{F}_{t_{n-1}} \otimes \mathcal{F'}_{t_{n-1}})(\omega,\omega').
 \end{align*}
Since  $\widetilde{L}(t)$ is a cylindrical L\'evy process, and for each $n \in \mathbb{N}$ we have $A_{n,k}\cap A_{n,l}=\emptyset$ for $k\neq l$ and $\cup_{k=1}^{M(n)}A_{n,k}=\Omega$, we obtain for all $h \in H$ and $n\in{\mathbb N}$ that
\begin{align} \label{EQNU}
     E_{P \otimes P'}&\Big[e^{i\langle Y_n, h\rangle}\Big\vert \mathcal{F}_{t_{n-1}} \otimes \mathcal{F'}_{t_{n-1}}\Big]\nonumber\\
     &=E_{P \otimes P'}\left[e^{i \left \langle \left(\sum_{k=1}^{M(n)}\mathbb{1}_{A_{n,k} \times \Omega'}\Phi_{n,k}\right)(\widetilde{L}(t_n)-\widetilde{L}(t_{n-1})), h \right\rangle}\Big\vert \mathcal{F}_{t_{n-1}} \otimes \mathcal{F'}_{t_{n-1}}\right]\nonumber\\
     &=\sum_{k=1}^{M(n)} E_{P \otimes P'}\Big[\mathbb{1}_{{A_{n,k}} \times \Omega'}\,e^{{i\langle \Phi_{n,k}(\widetilde{L}(t_n)-\widetilde{L}(t_{n-1}))}, h\rangle}\Big\vert \mathcal{F}_{t_{n-1}} \otimes \mathcal{F'}_{t_{n-1}}\Big]\nonumber\\
     &=\sum_{k=1}^{M(n)} \mathbb{1}_{{A_{n,k}} \times \Omega'}\, E_{P \otimes P'}\Big[e^{{i\langle \Phi_{n,k}(\widetilde{L}(t_n)-\widetilde{L}(t_{n-1}))}, h\rangle}\Big\vert \mathcal{F}_{t_{n-1}} \otimes \mathcal{F'}_{t_{n-1}}\Big]\nonumber\\
     &=\sum_{k=1}^{M(n)}\mathbb{1}_{{A_{n,k}} \times \Omega'}\, E_{P \otimes P'}\Big[e^{{i\langle \Phi_{n,k}(\widetilde{L}(t_n)-\widetilde{L}(t_{n-1}))}, h\rangle}\Big]\nonumber\\
     &=\sum_{k=1}^{M(n)}\mathbb{1}_{{A_{n,k}} \times \Omega'}\, E_{P'}\Big[e^{{i\langle \Phi_{n,k}({L}(t_n)-{L}(t_{n-1}))}, h\rangle}\Big]\nonumber\\
     &=\sum_{k=1}^{M(n)}\mathbb{1}_{{A_{n,k}}\times \Omega'}\,e^{{(t_n-t_{n-1}) S(\Phi_{n,k}^* h)}}\nonumber\\
    &= e^{{(t_n-t_{n-1})S(\Theta_n^* h)}} \quad P \otimes P'-\text{a.s.},
\end{align}
where $S$ denotes the cylindrical L\'evy symbol of $L$. In the same way we obtain 
\begin{align} \label{EQNT}
	E_{P \otimes P'}&\Big[e^{i\langle X_n, h\rangle}\Big\vert \mathcal{F}_{t_{n-1}} \otimes \mathcal{F'}_{t_{n-1}}\Big]
	= e^{{(t_n-t_{n-1}) S(\Theta_n^* h)}} \quad P \otimes P'-\text{a.s.}
\end{align}
It follows from  \eqref{EQNU} and \eqref{EQNT} by calculating the conditional expectation from the conditional probability, see e.g. \cite[Th.\ 6.4]{OK},  that for $P \otimes P'$ a.a.\ $(\omega,\omega') \in \Omega \times \Omega'$ and for all $u\in H$ we have
\begin{align*}
    \varphi_{(P \otimes P')_{X_n}(\cdot,(\omega,\omega'))}(u)&=\int_H e^{i\langle h,u\rangle}\, (P \otimes P')_{X_n}\big({\rm d}h,(\omega,\omega')\big)\\
    &=E_{P \otimes P'}\Big[e^{i\langle X_n, u\rangle}\Big\vert \mathcal{F}_{t_{n-1}} \otimes \mathcal{F'}_{t_{n-1}}\Big](\omega,\omega')\\
    &=E_{P \otimes P'}\Big[e^{i\langle Y_n, u\rangle}\Big\vert \mathcal{F}_{t_{n-1}} \otimes \mathcal{F'}_{t_{n-1}}\Big](\omega,\omega')\\
    &=\int_H e^{i\langle h,u\rangle}\,(P \otimes P')_{X_n}\big({\rm d}h,(\omega,\omega')\big)
    =\varphi_{(P \otimes P')_{Y_n}(\cdot,(\omega,\omega'))}(u).
\end{align*}
Since characteristic functions uniquely determine distributions on $\Borel(H)$, we obtain
\[(P \otimes P')_{X_n}(\cdot,(\omega,\omega'))=(P \otimes P')_{Y_n}(\cdot,(\omega,\omega')) \quad P \otimes P'-\text{a.s.}, \]
establishing Condition (2) of Definition \ref{dec_tang_seq_def}.
\end{proof}

\section{Characterisation of random integrands}\label{se.random-integrands}

The following is the main result of this paper characterising the largest space of predictable integrands which are stochastically integrable with respect to a cylindrical L\'evy process $L$ in a Hilbert space $G$. 
\begin{theorem} \label{pred_iff_integrable}
The space $\mathcal{I}_{{\rm prd},L}^{\rm HS}$ of  predicable Hilbert-Schmidt operator-valued processes integrable 
with respect to a cylindrical L\'evy process $L$ in $G$ coincides with the class of predictable processes in $L_P^0(\Omega, \mathcal{M}_{{\rm det},L}^{\rm HS})$.	
\end{theorem}
As in the case of deterministic integrands, the above characterisation of the space of $L$-integrable predictable processes 
strongly relies on the equivalent notion of convergences in two spaces. 
\begin{lemma} \label{le.small-if-small}
	Let $L$ be a cylindrical L\'evy process in $G$, and $(\Psi_n)_{n\in{\mathbb N}}$ a sequence in $\mathcal{S}_{\rm prd}^{\rm HS}$. Then the following are 
	equivalent:
	\begin{enumerate}
		\item[{\rm (a)}] $\displaystyle \lim_{n\to\infty}\normm{\Psi_n}_L=0$;
		\item[{\rm (b)}] $\displaystyle \lim_{n\to\infty} \displaystyle\sup_{\Gamma \in \mathcal{S}_{{\rm prd}}^{1, {\rm op}}}E\Bigg[\norm{\int_0^T \Gamma \Psi_n \;{\rm d}L}\wedge 1 \Bigg]=0$ and $\displaystyle\lim_{n \rightarrow \infty}E\left[m''(\Psi_n)\wedge 1\right]=0.$
	\end{enumerate}
\end{lemma}

\begin{proof}
To prove (a) $\Rightarrow$ (b), let $\epsilon>0$ be fixed. Lemma \ref{le.cont_of_int_op} and the fact that $m_L$ and $\rho_L$ generate the same topology on $\mathcal{M}_{{\rm det},L}^{\rm HS}$ enables us to choose 
$\delta>0$ such that for every $\psi\in\mathcal{S}_{\rm det}^{\rm HS}$ we have 
the implication:
\begin{align}\label{eq.det-small-if-small}
	\rho_L(\psi)\le \delta \;\Rightarrow\; 
 \sup_{\gamma \in \mathcal{S}^{1,{\rm op}}_{\rm det}}P\left(\norm{\int_0^T \gamma \psi  \;{\rm d}L}>\epsilon\right)\le 
 \epsilon. 	
\end{align}
Since $\lim_{n \rightarrow \infty}\normm{\Psi_n}_L=0$, there exists $n_0\in {\mathbb N}$ such that the set
\[A_n:=\left\{\omega \in \Omega: \rho_L(\Psi_n(\omega))\le \delta\right\}\]
satisfies $P(A_n)\geq 1-\epsilon$ for all $n\ge n_0$. By recalling the definition of $\widetilde{L}$ and $(\Omega',\mathcal{F}',P')$ from Proposition \ref{dec_tan_seq}, implication 
\eqref{eq.det-small-if-small} implies for all $\omega\in A_n$ and  $n\geq n_0$ that 
\[
\sup_{\Gamma \in \mathcal{S}_{{\rm prd}}^{1, {\rm op}}}
P'\left(\omega'\in \Omega':\norm{\left(\int_0^T \Gamma(\omega)\Psi_n(\omega) \;{\rm d}\widetilde{L}(\omega,\cdot)\right)(\omega^\prime)}> \epsilon \right)\leq\epsilon.\]
Fubini's theorem implies for all $n\geq n_0$ and $\Gamma \in \mathcal{S}_{{\rm prd}}^{1, {\rm op}}$ that
\begin{align*}
    &(P\otimes P')\left((\omega, \omega')\in \Omega\times \Omega'\colon   \norm{\left(\int_0^T \Gamma \Psi_n \;{\rm d}\widetilde{L}\right)(\omega,\omega')}> \epsilon\right)\\
    &=\int_{\Omega} P'\left(\omega'\in \Omega':  \norm{\left(\int_0^T \Gamma(\omega) \Psi_n(\omega) \;{\rm d}\widetilde{L}(\omega,\cdot)\right)(\omega')}> \epsilon\right)\,P({\rm d}\omega)
    \leq 2\epsilon
\end{align*}
As $\epsilon>0$ is arbitrary, we obtain
\begin{equation}\label{limit_sup}
    \lim_{n \rightarrow \infty}\sup_{\Gamma \in \mathcal{S}_{{\rm prd}}^{1, {\rm op}}}E_{P \otimes P'} \Bigg[\norm{\int_0^T \Gamma \Psi_n \;{\rm d}\widetilde{L}} \wedge 1\Bigg] = 0.
\end{equation}
By the ideal property of $L_2(G,H)$, for each $n \in \mathbb{N}$ and $\Gamma \in \mathcal{S}_{{\rm prd}}^{1, {\rm op}}$ the integrand $\Gamma \Psi_n$ lies in $\mathcal{S}_{\rm prd}^{\rm HS}$ and has a representation of the from
\begin{align}\label{eq.presentiation-GammaS}
\Gamma \Psi_n=\Gamma_0^n\Phi_0^n \mathbb{1}_{\{0\}}+\sum_{i=1}^{N(n)-1} \Gamma_i^n\Phi_i^n \mathbb{1}_{(t_i^n,t_{i+1}^n]},
\end{align}
where $0=t_1^n\leq \cdots <t_{N(n)}^n= T$, and  $\Gamma_i^n\Phi_i^n$ is an $\mathcal{F}_{t_i^n}$-measurable $L_2(G,H)$-valued random variable taking only finitely many values  for each $i = 0,...,N(n)-1$.
Proposition \ref{dec_tan_seq} guarantees for each $n \in \mathbb{N}$ that the sequence of Radonified increments
\[\Big(\Gamma_i^n\Phi_i^n(L(t_{i+1}^n)-L(t_i^n))\Big)_{i=1,...,N_n-1}\]
has the decoupled tangent sequence
\[\Big(\Gamma_i^n\Phi_i^n (\widetilde{L}(t_{i+1}^n)-\widetilde{L}(t_i^n))\Big)_{i=1,...,N_n-1}.\]
We conclude from  the decoupling inequality \cite[Pr.\ 5.7.1.(ii)]{kwapien_woyczynski_1992} that there exists a constant $c>0$ such that, for all $n \in \mathbb{N}$ and $\Gamma \in \mathcal{S}_{{\rm prd}}^{1, {\rm op}}$, we have
\begin{align*}
E_{P\otimes P'}\Bigg[\norm{\int_0^T \Gamma \Psi_n \;{\rm d}L}\wedge 1\Bigg]
&= E_{P\otimes P'}\Bigg[\norm{\sum_{i=1}^{N(n)-1} \Gamma_i^n\Phi_i^n (L(t_{i+1}^n)-L(t_i^n))}\wedge 1\Bigg] \nonumber\\
& \leq c E_{P\otimes P'}\Bigg[\norm{\sum_{i=1}^{N(n)-1} \Gamma_i^n\Phi_i^n (\widetilde{L}(t_{i+1}^n)-\widetilde{L}(t_i^n))} \wedge 1\Bigg]\nonumber \\
&=c E_{P\otimes P'}\Bigg[\norm{\int_0^T \Gamma \Psi_n \;{\rm d}\widetilde{L}}\wedge 1\Bigg].
\end{align*}
We conclude from Remark \ref{product_expectation_remark} and \eqref{limit_sup} that 
\begin{equation*} 
 \lim_{n \rightarrow \infty}\sup_{\Gamma \in \mathcal{S}_{{\rm prd}}^{1, {\rm op}}} E_{P}\Bigg[\norm{\int_0^T \Gamma \Psi_n \;{\rm d}L}\wedge 1\Bigg]= \lim_{n \rightarrow \infty}\sup_{\Gamma \in \mathcal{S}_{{\rm prd}}^{1, {\rm op}}} E_{P\otimes P'}\Bigg[\norm{\int_0^T \Gamma \Psi_n \;{\rm d}L}\wedge 1\Bigg]=0. 
\end{equation*}
Seeing that $m_L$ and $\rho_L$ generate the same topology on $\mathcal{M}_{{\rm det},L}^{\rm HS}$, our assumption that $\lim_{n \rightarrow \infty}\normm{\Psi_n}_L=0$ implies $\lim_{n \rightarrow \infty}E[m''(\Psi_n)\wedge 1]=0$, which immediately gives (b).

For establishing (b) $\Rightarrow$ (a), given any $\Gamma \in \mathcal{S}_{{\rm prd}}^{1, {\rm op}}$ we may assume that $\Gamma \Psi_n$ has a representation of the form \eqref{eq.presentiation-GammaS}. We conclude from \cite[Pr.\ 5.7.2]{kwapien_woyczynski_1992} that there exists a constant $c>0$ such that for all $\Gamma \in \mathcal{S}_{{\rm prd}}^{1, {\rm op}}$ we have
\begin{align}
    E_{P \otimes P'}\Bigg[\norm{\int_0^T \Gamma \Psi_n \;{\rm d}\widetilde{L}}\wedge 1\Bigg] 
    &= E_{P \otimes P'}\Bigg[\norm{\sum_{i=1}^{N(n)-1}\Gamma_i^n\Phi_i^n(\widetilde{L}(t_{i+1}^n)-\widetilde{L}(t_i^n))}\wedge 1\Bigg] \nonumber\\
    &\leq c \max_{\epsilon_i \in \{\pm1\}} E_{P \otimes P'}\Bigg[\norm{\sum_{i=1}^{N(n)-1}\epsilon_i \Gamma_i^n\Phi_i^n(L(t_{i+1}^n)-L(t_i^n))}\wedge 1\Bigg] \nonumber\\
    &=c \max_{\epsilon_i \in \{\pm1\}} E_{P}\Bigg[\norm{\sum_{i=1}^{N(n)-1}\epsilon_i\Gamma_i^n \Phi_i^n(L(t_{i+1}^n)-L(t_i^n))}\wedge 1\Bigg] \nonumber\\
    &\leq c \sup_{\Theta \in \mathcal{S}_{{\rm prd}}^{1, {\rm op}}} E_{P}\Bigg[\norm{\sum_{i=1}^{N(n)-1}\Theta_i^n\Phi_i^n(L(t_{i+1}^n)-L(t_i^n))}\wedge 1\Bigg] \nonumber\\
    &=c \sup_{\Theta \in \mathcal{S}_{{\rm prd}}^{1, {\rm op}}} E_{P}\Bigg[\norm{\int_0^T \Theta \Psi_n \;{\rm d}L}\wedge 1\Bigg].\label{main_proof_decoupl_ineq}
\end{align}
By choosing $\Gamma=\text{Id}\mathbb{1}_{\Omega \times (0,T]}$, the hypothesis on $(\Psi_n)_{n\in{\mathbb N}}$ implies 
 \[\lim_{n \rightarrow \infty} E_{P \otimes P'}\Bigg[\norm{\int_0^T \Psi_n \;{\rm d}\widetilde{L}}\wedge 1\Bigg] = 0. \]
It follows that for every subsequence $(\Psi_{n_m})_{m \in \mathbb{N}}$ of $(\Psi_n)_{n \in \mathbb{N}}$, there exists a further subsequence $(\Psi_{n_{m_j}})_{j \in \mathbb{N}}$ and a set $N \subseteq \Omega \times \Omega'$ with $(P \otimes P')(N)=0$ satisfying
\[\lim_{j \rightarrow \infty}\Bigg(\int_0^T \Psi_{n_{m_j}}\;{\rm d}\widetilde{L}\Bigg)(\omega,\omega')=0 \quad \text{for each}\; (\omega,\omega')\in N^c.\]
Define the section of the set $N$ for each $\omega \in \Omega$  by
\[N_{\omega}=\Bigg\{\omega' \in \Omega'\colon  \lim_{j \rightarrow \infty}\left(\int_0^T \Psi_{n_{m_j}}(\omega)\;{\rm d} \widetilde{L}(\omega,\cdot)\right)(\omega')\neq 0\Bigg\},\]
where we note that since $\Psi_{n_{m_j}}$ are step processes, it holds that
\[\Bigg(\int_0^T \Psi_{n_{m_j}}\;{\rm d}\widetilde{L}\Bigg)(\omega,\cdot)
=\int_0^T \Psi_{n_{m_j}}(\omega) \;{\rm d} \widetilde{L}(\omega,\cdot)\qquad\text{for all }\omega\in\Omega .\]
Fubini's theorem implies $0=(P \otimes P')(N)= \int_{\Omega}P'(N_{\omega}){\rm d}P(\omega)$,  from which it follows that there exists  $\Omega_1 \subseteq \Omega$ with $P(\Omega_1)=1$ such that $P'(N_{\omega})=0$  for all $\omega \in \Omega_1$. In other words, for each fixed $\omega \in \Omega_1$, the sequence of random variables
\[\Bigg(\int_0^T \Psi_{n_{m_j}}(\omega) \;{\rm d} \widetilde{L}(\omega,\cdot)\Bigg )_{j \in \mathbb{N}}\]
converges $P'$-a.s.\ to $0$ as $H$-valued random variables on $(\Omega',\mathcal{F}',P')$. 
For each fixed $\omega \in \Omega_1$, the above sequence is infinitely divisible and has characteristics
\[\left(\int_0^T b_{{\Psi_{n_{m_j}}}(\omega,t)}^\theta\,{\rm d}t, \,\int_0^T\Psi_{n_{m_j}}(\omega,t) Q \Psi_{n_{m_j}}^*(\omega,t)\,{\rm d}t, \, (\lambda \otimes {\rm Leb}) \circ \kappa_{\Psi_{n_{m_j}}(\omega)}^{-1}\right),\]
by Lemmata \ref{le.measure_theoretic_lemma} and \cite[Le. 3.6]{BR} where $\kappa_{\Psi}\colon G\times [0,T]\to H$ is defined by $\kappa_{\psi}(g,t)=\psi(t)g$ for $\psi\in \mathcal{S}_{\rm det}^{\rm HS}$.  Since the cylindrical L\'evy process $\widetilde{L}(\omega, \cdot)$ has the same cylindrical characteristics as $L$ for each $\omega \in \Omega$, we obtain for all $\omega \in \Omega_1$ that
\begin{align*}
    \lim_{j \rightarrow \infty}k_{L}(\Psi_{n_{m_j}}(\omega))=\lim_{j \rightarrow \infty}k_{\widetilde L(\omega,\cdot)}(\Psi_{n_{m_j}}(\omega))=0.
\end{align*}
As  $P(\Omega_1)=1$ , the above argument proves that for all $\epsilon>0$ we have
\begin{align}\label{eq.small_k_main_proof}
    \lim_{n \rightarrow \infty}P\left(\int_0^T k_L(\Psi_n(t))\, {\rm d}t >\epsilon\right)=0.
\end{align}
To finish the proof, it remains to show that for all $\epsilon>0$ we have
\begin{align}\label{eq.small_l_main_proof}
    \lim_{n \rightarrow \infty}P\left(\int_0^T l_L(\Psi_n(t))\, {\rm d}t>\epsilon\right)=0.
\end{align}
Let $\epsilon \in (0,1)$ be fixed. Since stochastic integrals with deterministic integrands with respect to $L$ are infinitely divisible, Remark \ref{re.inf_div_continuity_first_char} implies that there exists $\delta\in(0,\epsilon)$ such that for all $\psi \in \mathcal{M}_{\rm det}^{\rm HS}$ we have the implication
\begin{align}\label{eq.implication_characteristic_small}
    P'\left(\norm{\int_0^T \psi\,{\rm d}L}>\sqrt{\delta}\right)<\sqrt{\delta} \implies \norm{\int_0^T b_{\psi(t)}^\theta\, {\rm d}t}<\epsilon.
\end{align}
It follows from Equation \eqref{main_proof_decoupl_ineq} that there exists an $N \in \mathbb{N}$ such that for all $n \geq N$ we have
\begin{align}\label{eq.product_measure_goes_to_0}
    \sup_{\Gamma \in \mathcal{S}_{{\rm prd}}^{1, {\rm op}}}(P\otimes P')\left(\norm{\int_0^T \Gamma \Psi_n \;{\rm d}\widetilde{L}}>\sqrt{\delta}\right)<\delta.
\end{align}
Chebyshev's inequality, Fubini's theorem and Equation \eqref{eq.product_measure_goes_to_0} imply for all $n \geq N$ and $\Gamma \in \mathcal{S}_{{\rm prd}}^{1, {\rm op}}$ that
\begin{align}\label{eq.one_minus_delta_bound}
    &P\left(\omega \in \Omega:P'\left(\omega' \in \Omega':\norm{\left(\int_0^T \Gamma(\omega) \Psi_n(\omega) \;{\rm d}\widetilde{L}(\omega,\cdot)\right)(\omega')}>\sqrt{\delta}\right)<\sqrt{\delta}\right)\nonumber\\
    &\qquad= 1- P\left(\omega \in \Omega:P'\left(\omega' \in \Omega':\norm{\left(\int_0^T \Gamma(\omega) \Psi_n(\omega) \;{\rm d}\widetilde{L}(\omega,\cdot)\right)(\omega')}>\sqrt{\delta}\right)\geq\sqrt{\delta}\right)\nonumber\\
    &\qquad\geq 1 - \frac{1}{\sqrt{\delta}}\int_{\Omega}P'\left(\omega' \in \Omega':\norm{\left(\int_0^T \Gamma(\omega) \Psi_n(\omega) \;{\rm d}\widetilde{L}(\omega,\cdot)\right)(\omega')}>\sqrt{\delta}\right)\,{\rm d}P(\omega)\nonumber\\
    &\qquad= 1 - \frac{1}{\sqrt{\delta}}(P\otimes P')\left((\omega,\omega')\in \Omega \times \Omega':\norm{\left(\int_0^T \Gamma \Psi_n \;{\rm d}\widetilde{L}\right)(\omega,\omega')}>\sqrt{\delta}\right)\nonumber\\
    &\qquad\geq 1- \sqrt{\delta}.
\end{align}
In light of Equations \eqref{eq.implication_characteristic_small} and \eqref{eq.one_minus_delta_bound}, we have for all $n \geq N$ and $\Gamma \in \mathcal{S}_{{\rm prd}}^{1, {\rm op}}$ that
\begin{align*}
    P\left(\norm{\int_0^T b_{\Gamma \Psi_n}^\theta\, {\rm d}t}<\epsilon\right)\geq 1-\sqrt{\delta},
\end{align*}
or equivalently, for all $n \geq N$ we have
\begin{align*}
    \sup_{\Gamma \in \mathcal{S}_{{\rm prd}}^{1, {\rm op}}}P\left(\norm{\int_0^T b_{\Gamma \Psi_n}^\theta\, {\rm d}t}\geq\epsilon\right)\leq \sqrt{\delta}.
\end{align*}
The above inequality, combined with an approximation argument using functions in $\mathcal{S}_{{\rm prd}}^{1, {\rm op}}$ shows that for any predictable $\Bar{B}_{L(H)}$-valued process $\Lambda$ and $n \geq N$ it holds that
\begin{align}\label{eq.general_prdictable_probability}
    P\left(\norm{\int_0^T b_{\Lambda \Psi_n}^\theta\, {\rm d}t}\geq\epsilon\right)\leq \sqrt{\delta}.
\end{align}
For each fixed $n \geq N$, define a process $H_n:\Omega \times [0,T]\rightarrow \Bar{B}_{L(H)}$ by
\[H_n(\omega,t)=f\left(b_{i(\Psi_n(\omega,t))\Psi_n(\omega,t)}^\theta\right)\big( i(\Psi_n(\omega,t))\big),\]
with $i$ and $f$ as in the proof of Lemma \ref{le.supremum_equivalency}. Then, $H_n$ is predictable and, by the same argument as in Lemma \ref{le.supremum_equivalency}, it satisfies for each $\omega \in \Omega$ that
\begin{align*}
    \int_0^T \sup_{O \in \Bar{B}_{L(H)}} \norm{b_{O \Psi_n(\omega,t)}^\theta}\, {\rm d}t\leq \norm{\int_0^T b_{H_n(\omega,t) \Psi_n(\omega,t)}^\theta\, {\rm d}t}+\epsilon.
\end{align*}
Applying Equation \eqref{eq.general_prdictable_probability} for $\Lambda=H_n$ shows for all $n \geq \mathbb{N}$ that
\[P\left(\int_0^T \sup_{O \in \Bar{B}_{L(H)}} \norm{b_{O \Psi_n(t)}^\theta}\, {\rm d}t\geq2\epsilon\right)\leq P\left(\norm{\int_0^T b_{H_n(t) \Psi_n(t)}^\theta\, {\rm d}t}\geq\epsilon\right)\leq \sqrt{\delta}.\]
Since we have that $\delta<\epsilon$, this finishes the proof of the claim in Equation \eqref{eq.small_l_main_proof}. Finally, by Equations \eqref{eq.small_k_main_proof}, \eqref{eq.small_l_main_proof}, and the assumption that $\lim_{n \rightarrow \infty}E\left[m''(\Psi_n)\wedge 1\right]=0$, we obtain that $\lim_{n \rightarrow \infty}E\left[m_L(\Psi_n)\wedge 1\right]=0$. This completes the proof, since $m_L$ and $\rho_L$ generate the same topology.
\end{proof}

\begin{proof} [Proof of Theorem \ref{pred_iff_integrable}.] 
	If $\Psi\in \mathcal{I}_{{\rm prd},L}^{\rm HS}$ then Definition \ref{pred_integrability} guarantees the existence of a sequence $(\Psi_n)_{n \in \mathbb{N}}$ of elements of $\mathcal{S}_{\rm prd}^{\rm HS}$ converging $P_T$-a.e. to $\Psi$ and satisfying 
	\[\displaystyle \lim_{m,n \rightarrow \infty}\sup_{\Gamma \in \mathcal{S}_{{\rm prd}}^{1, {\rm op}}}E\Bigg[\norm{\int_0^T \Gamma(\Psi_m-\Psi_n) \;{\rm d}L}\wedge1 \Bigg]=0.\]
	 Lemma \ref{le.small-if-small} implies that $\lim_{m,n \rightarrow \infty}\normm{\Psi_m-\Psi_n}_L=0$. Completeness of the metric space $(L_P^0(\Omega, \mathcal{M}_{{\rm det},L}^{\rm HS}),\normm{\cdot}_L)$ and the fact that $(\Psi_n)_{n \in \mathbb{N}}$ converges $P_T$-a.e. to $\Psi$ together yield that the sequence $(\Psi_n)_{n \in \mathbb{N}}$ has a limit in $L_P^0(\Omega, \mathcal{M}_{{\rm det},L}^{\rm HS})$ and that this limit necessarily coincides with $\Psi$. Thus $\Psi\in L_P^0(\Omega, \mathcal{M}_{{\rm det},L}^{\rm HS})$. 
	
	To establish the reverse inclusion, let $\Psi$ be a predictable process in the space $L_P^0(\Omega, \mathcal{M}_{{\rm det},L}^{\rm HS})$. Lemma \ref{pred_density} guarantees that there exists a sequence $(\Psi_n)_{n \in \mathbb{N}}$ of elements of $\mathcal{S}_{\rm prd}^{\rm HS}$ converging to $\Psi$ in $\normm{\cdot}_L$ and $P_T$-a.e. Then, $(\Psi_m-\Psi_n)$ converges to $0$ both in $\normm{\cdot}_L$ and $P_T$-a.e. as $m,n \rightarrow \infty$. This implies by Lemma \ref{le.small-if-small} that
	\[\displaystyle \lim_{m,n \rightarrow \infty}\sup_{\Gamma \in \mathcal{S}_{{\rm prd}}^{1, {\rm op}}}E\Bigg[\norm{\int_0^T \Gamma(\Psi_m-\Psi_n) \;{\rm d}L}\wedge1 \Bigg]=0.\]
	Thus $\Psi$ satisfies the conditions of Definition \ref{pred_integrability}, which means that $\Psi \in \mathcal{I}_{{\rm prd},L}^{\rm HS}$.
\end{proof}

Lemma  \ref{le.small-if-small} is crucial to characterise the space of integrable predictable processes in Theorem \ref{pred_iff_integrable}, as it describes convergence of predictable step processes in the space of integrands in terms of convergence in the randomised modular space. Having identified the space of integrable predictable processes, we can extend Lemma \ref{le.small-if-small} to the whole space of integrable predictable processes.

\begin{corollary} \label{co.small-if-small}
	Let $L$ be a cylindrical L\'evy process in $G$, and $(\Psi_n)_{n\in{\mathbb N}}$ a sequence in $\mathcal{I}_{{\rm prd},L}^{\rm HS}$. Then the following are 
	equivalent:
	\begin{enumerate}
		\item[{\rm (a)}] $\displaystyle \lim_{n\to\infty}\normm{\Psi_n}_L=0$; 
		\item[{\rm (b)}] $\displaystyle \lim_{n\to\infty} \displaystyle\sup_{\Gamma \in \mathcal{S}_{{\rm prd}}^{1, {\rm op}}}E\Bigg[\norm{\int_0^T \Gamma \Psi_n \;{\rm d}L}\wedge 1 \Bigg]=0$ and $\displaystyle\lim_{n \rightarrow \infty}E\left[m''(\Psi_n)\wedge 1\right]=0$. 
	\end{enumerate}
\end{corollary}
\begin{proof}
 To establish the implication (a) $\Rightarrow$ (b), first note that it follows from the definition of $\normm{\cdot}_L$ and the fact that $\rho_L$ generate the same topology as $m_L$ that $m''(\Psi_n)\rightarrow 0$ in probability. Let $\epsilon>0$ be fixed. Lemma \ref{le.small-if-small} implies that there exists a $\delta(\epsilon)>0$ such that we have for all $\Psi \in \mathcal{S}_{\rm prd}^{\rm HS}$ the implication:
 \begin{equation}\label{epsilon_delta_implication}
     \normm{\Psi}_L<\delta(\epsilon) \quad \Rightarrow \quad \displaystyle\sup_{\Gamma \in \mathcal{S}_{{\rm prd}}^{1, {\rm op}}}E\Bigg[\norm{\int_0^T \Gamma \Psi \;{\rm d}L}\wedge 1 \Bigg]<\epsilon.
 \end{equation}
 Since $\lim_{n\to\infty}\normm{\Psi_n}_L=0$, there exists an $n_0 \in \mathbb{N}$ such that $\normm{\Psi_n}_L<\tfrac{\delta(\epsilon)}{2}$  for all $n \geq n_0$. By Theorem \ref{pred_iff_integrable} we have that $(\Psi_n)_{n \in \mathbb{N}} \subseteq L_P^0(\Omega, \mathcal{M}_{{\rm det},L}^{\rm HS})$, hence Lemma \ref{pred_density} guarantees for each $n \in \mathbb{N}$ the existence of a sequence $(\Psi_n^m)_{m \in \mathbb{N}}\subseteq \mathcal{S}_{\rm prd}^{\rm HS}$ converging to $\Psi_n$ in $\normm{\cdot}_L$ and $P_T$-a.e. Consequently, we can find $m_0(n,\epsilon) \in \mathbb{N}$ for each $n \in \mathbb{N}$ such that for all $m \geq m_0(n,\epsilon)$ we have 
$\normm{\Psi_n^m-\Psi_n}_L< \tfrac{\delta(\epsilon)}{2}$.
We obtain for each $n\geq n_0$ and $m \geq m_0(n,\epsilon)$ that
 \[\normm{\Psi_n^m}_L\leq \normm{\Psi_n^m-\Psi_n}_L+\normm{\Psi_n}_L< \delta(\epsilon),\]
which implies by (\ref{epsilon_delta_implication}) that
 \begin{equation}\label{eq_small_approx_int}
     \displaystyle\sup_{\Gamma \in \mathcal{S}_{{\rm prd}}^{1, {\rm op}}}E\Bigg[\norm{\int_0^T \Gamma \Psi_n^m \;{\rm d}L}\wedge 1 \Bigg]<\epsilon.
 \end{equation}
Thus, if we fix an $n\geq n_0$ and recall that the integral of $\Psi_n$ is defined to be the limit in probability of the integrals of $\Psi_n^m$ as $m\to\infty$, we obtain from Equation (\ref{eq_small_approx_int}) that
\begin{align*}
         \displaystyle\sup_{\Gamma \in \mathcal{S}_{{\rm prd}}^{1, {\rm op}}}E\Bigg[\norm{\int_0^T \Gamma \Psi_n \;{\rm d}L}\wedge 1 \Bigg] &= \sup_{\Gamma \in \mathcal{S}_{{\rm prd}}^{1, {\rm op}}} \lim_{m \rightarrow \infty}E\Bigg[\norm{\int_0^T \Gamma \Psi_n^m \;{\rm d}L}\wedge 1 \Bigg] \\
         &\leq \limsup_{m \rightarrow \infty} \sup_{\Gamma \in \mathcal{S}_{{\rm prd}}^{1, {\rm op}}}E\Bigg[\norm{\int_0^T \Gamma \Psi_n^m \;{\rm d}L}\wedge 1 \Bigg]< \epsilon.
\end{align*}
 
To establish the reverse implication (b) $\Rightarrow$ (a), let $\epsilon>0$ be fixed. Lemma \ref{le.small-if-small} implies that there exists a $\delta(\epsilon)>0$ such that we have for all $\Psi \in \mathcal{S}_{\rm prd}^{\rm HS}$ the implication: 
 \begin{equation}\label{reverse_epsilon_delta_implication}
     \displaystyle\sup_{\Gamma \in \mathcal{S}_{{\rm prd}}^{1, {\rm op}}}E\Bigg[\norm{\int_0^T \Gamma \Psi \;{\rm d}L}\wedge 1 \Bigg]+E\left[m''(\Psi)\wedge 1\right]<\delta(\epsilon) \quad \Rightarrow \quad \normm{\Psi}_L<\tfrac{\epsilon}{2}.
 \end{equation}
 By assumption, there exists an $n_0 \in \mathbb{N}$ such that for all $n \geq n_0$ we have
 \begin{equation} \label{eq.small_integral}
      \displaystyle\sup_{\Gamma \in \mathcal{S}_{{\rm prd}}^{1, {\rm op}}}E\Bigg[\norm{\int_0^T \Gamma \Psi_n \;{\rm d}L}\wedge 1 \Bigg]+E\left[m''(\Psi_n)\wedge 1\right]<\tfrac{\delta(\epsilon)}{4}.
 \end{equation}
 As $(\Psi_n)_{n \in \mathbb{N}} \subseteq \mathcal{I}_{{\rm prd},L}^{\rm HS}$, it follows from Theorem \ref{pred_iff_integrable} and Lemma \ref{pred_density} that for each $n \in \mathbb{N}$  there exists a sequence $(\Psi_n^m)_{m \in \mathbb{N}}$ of elements of $\mathcal{S}_{\rm prd}^{\rm HS}$ converging to $\Psi_n$ in $\normm{\cdot}_L$ and $P_T$-a.e. Consequently,  we can find $m_0(n,\epsilon) \in \mathbb{N}$  for each $n \in \mathbb{N}$, such that for all $m \geq m_0(n,\epsilon)$ we have
 \begin{equation} \label{eq.process_dif_is_small}
     \normm{\Psi_n^m-\Psi_n}_L< \epsilon/2.
 \end{equation}
Since for each $n \in \mathbb{N}$ we have that $\lim_{m \rightarrow \infty}\normm{\Psi_n^m-\Psi_n}_L=0$, the first part of this Corollary and the reverse triangle inequality shows that for each $n \in \mathbb{N}$ there exists an $m_1(n,\epsilon) \in \mathbb{N}$ such that for all $m \geq m_1(n,\epsilon)$ we have
 \begin{align}\label{eq.integrals_close}
     \left \vert\displaystyle\sup_{\Gamma \in \mathcal{S}_{{\rm prd}}^{1, {\rm op}}}E\Bigg[\norm{\int_0^T \Gamma \Psi_n \;{\rm d}L}\wedge 1 \Bigg]-\displaystyle\sup_{\Gamma \in \mathcal{S}_{{\rm prd}}^{1, {\rm op}}}E\Bigg[\norm{\int_0^T \Gamma \Psi_n^m \;{\rm d}L}\wedge 1 \Bigg]\right\vert<\tfrac{\delta(\epsilon)}{4}.
 \end{align}
Moreover, since for each $n \in \mathbb{N}$ we have that $\Psi_n^m \rightarrow \Psi_n$ $P_T$-a.e.\ as $m \rightarrow \infty$, there exists $m_2(n,\epsilon)\in \mathbb{N}$ such that for all $m\geq m_2(n,\epsilon)$ it holds that
\begin{align}\label{eq.control_for_m''}
    \left \vert E\left[m''(\Psi_n)\wedge 1\right]-E\left[m''(\Psi_n^m)\wedge 1\right] \right \vert<\frac{\delta(\epsilon)}{4}.
\end{align}
By combining Equations (\ref{eq.small_integral}),(\ref{eq.integrals_close}) and \eqref{eq.control_for_m''}, we obtain for all $n\geq n_0$ and $m \geq \max\{m_0(n,\epsilon),m_1(n,\epsilon), m_2(n,\epsilon)\}$ that
 \[\sup_{\Gamma \in \mathcal{S}_{{\rm prd}}^{1, {\rm op}}}E\Bigg[\norm{\int_0^T \Gamma \Psi_n^m \;{\rm d}L}\wedge 1 \Bigg]+E\left[m''(\Psi_n^m)\wedge 1\right]<\delta(\epsilon),\]
 which implies by (\ref{reverse_epsilon_delta_implication}) and (\ref{eq.process_dif_is_small}) that
 \[\normm{\Psi_n}_L\leq \normm{\Psi_n-\Psi_n^m}_L+\normm{\Psi_n^m}_L<\epsilon. \]
 As $\epsilon>0$ was arbitrary, this concludes the proof.
\end{proof}

Having introduced the notion of the stochastic integral, we now show that stochastic integral processes, obtained by fixing an integrand and varying the upper limit of the stochastic integral, are in fact semimartingales.

\begin{theorem}\label{th_semimartingale}
If $\Psi \in \mathcal{I}_{{\rm prd},L}^{\rm HS}$, then the integral process $\left(I(\Psi)(t): t \in [0,T]\right)$ defined  by
\begin{align*}
    I(\Psi)(t):=\int_0^T \1_{[0,t]}(s) \Psi(s)\, L({\rm d} s)
    \qquad\text{for }t\in [0,T], 
\end{align*}
is a semimartingale.
\end{theorem}
\begin{proof}
Let $\Gamma \in \mathcal{S}_{\rm prd}^{1,\rm{op}}$ be of the form 
\[\Gamma(t)=\Gamma_0\mathbb{1}_{\{0\}}(t)+\sum_{k=1}^{N-1} \Gamma_k \mathbb{1}_{(s_k,s_{k+1}]}(t),\]
where $0=s_1<...<s_N=T$ are deterministic times, $\Gamma_0:\Omega \rightarrow \Bar{B}_{L(H)}$ is $\mathcal{F}_0$-measurable, and each $\Gamma_k:\Omega \rightarrow \Bar{B}_{L(H)}$ is an $\mathcal{F}_{s_k}$-measurable random variable taking only finitely many values for $k=1,...,N-1$.
Then we define the stochastic integral
\[\int_0^T \Gamma \, {\rm d}I(\Psi):=\sum_{k=1}^{N-1}\Gamma_k\bigg(I(\Psi)(s_{k+1})-I(\Psi)(s_{k})\bigg).\]
To prove the claim it suffices to show that the set
$
    \big\{ \int_0^T \Gamma \, {\rm d}I(\Psi) : \Gamma \in \mathcal{S}_{\rm prd}^{1,\rm{op}} \big\}
$
is bounded in probability according to \cite[Th.\ 2.1]{radonif_by_single}. Suppose, aiming for a contradiction, that it is not the case. Then there exists an $\epsilon>0$ and a sequence $(\Gamma_n)_{n \in \mathbb{N}}\subseteq \mathcal{S}_{\rm prd}^{1,\rm{op}}$ satisfying for all $n \in \mathbb{N}$ that
\begin{align}\label{eq.contradiction}
    P\left(\norm{\int_0^T \Gamma_n \, {\rm d}I(\Psi)}>n\right)\geq \epsilon.
\end{align}
For each $\Psi \in \mathcal{S}_{\rm prd}^{\rm HS}$ and $\Gamma \in \mathcal{S}_{\rm prd}^{1,\rm{op}}$, the very definitions of stochastic integrals show 
\begin{equation*}
    \int_0^T \Gamma \, {\rm d}I(\Psi) = \int_0^T \Gamma \Psi \, {\rm d}L. 
\end{equation*}
This equality can be generalised to arbitrary $\Psi \in \mathcal{I}_{{\rm prd},L}^{\rm HS}$ and $\Gamma \in \mathcal{S}_{\rm prd}^{1,\rm{op}}$ by a standard approximation argument. Using this to rewrite Equation (\ref{eq.contradiction}), we obtain for all $n \in \mathbb{N}$ that
\begin{align}\label{eq.bounded_set}
    \epsilon \leq P\left(\norm{\int_0^T \Gamma_n \, {\rm d}I(\Psi)}>n\right) 
    = P\left(\norm{\int_0^T \frac{1}{n}\Gamma_n \Psi \, {\rm d}L}>1\right).
\end{align}
On the other hand, since $\normm{\frac{1}{n}\Gamma_n \Psi}_L \rightarrow 0$ as $n \rightarrow \infty$,  Corollary \ref{co.small-if-small} implies
\begin{align*}
    \lim_{n \rightarrow \infty}E\left[ \norm{\int_0^T \frac{1}{n}\Gamma_n \Psi \, {\rm d}L} \wedge 1 \right]=0,
\end{align*}
which contradicts (\ref{eq.bounded_set}) because of the equivalent characterisation of the topology in $L^0_P(\Omega,H)$.
\end{proof}

We finish this section with a stochastic dominated convergence theorem.
\begin{theorem} \label{stoch_dom_conv}
	Let $(\Psi_n)_{n \in \mathbb{N}}$ be a sequence of processes in $\mathcal{I}_{{\rm prd},L}^{\rm HS}$ such that
	\begin{enumerate}[\rm (1)]
		\item \label{thm_cnd1'} $(\Psi_n)_{n \in \mathbb{N}}$ converges $P_T$-a.e.\ to an $L_2(G,H)$-valued predictable process $\Psi$;
		\item \label{thm_cnd2'} there exists a process $\Upsilon \in \mathcal{I}_{{\rm prd},L}^{\rm HS}$ satisfying for all $n \in \mathbb{N}$ that
		\[(k_L+l_L)(\Psi_n(\omega,t)) \leq (k_L+l_L)(\Upsilon(\omega,t)) \quad  \text{for $P_T$-a.a.\ $(\omega,t)\in \Omega\times [0,T]$.} \]
	\end{enumerate}
	Then it follows that $\Psi \in \mathcal{I}_{{\rm prd},L}^{\rm HS}$ and 
	\[\displaystyle \lim_{n\to\infty} \displaystyle\sup_{\Gamma \in \mathcal{S}_{{\rm prd}}^{1, {\rm op}}}E\Bigg[\norm{\int_0^T \Gamma (\Psi_n-\Psi) \;{\rm d}L}\wedge 1 \Bigg]=0.\]
\end{theorem}
\begin{proof}
By assumption, there exists a set $N \subseteq \Omega \times [0,T]$ with $P_T(N)=0$ such that  $\lim_{n \rightarrow \infty} \Psi_n(\omega,t)=\Psi(\omega,t)$ and  $(k_L+l_L)(\Psi_n(\omega,t)) \leq (k_L+l_L)(\Upsilon(\omega,t))$ for all $(\omega,t)\in N^c$ and $n \in \mathbb{N}$. Fubini's theorem yields that
\[0=P_T(N)=\int_\Omega {\rm Leb}\vert_{[0,T]}(N_\omega)\,P({\rm d}\omega),\]
where
\begin{align*}
    &N_\omega\!:=\!
   \left\{t\in [0,T]: \lim_{n \rightarrow \infty} \Psi_n(\omega,t)\neq\Psi(\omega,t)\,\text{or}\, (k_L+l_L)(\Psi_n(\omega,t)) > (k_L+l_L)(\Upsilon(\omega,t))\right\}.
\end{align*}
It follows that there exists an $\Omega_1 \subseteq \Omega$ with $P(\Omega_1)=1$ such that ${\rm Leb}\vert_{[0,T]}(N_{\omega})=0$ for all $\omega \in \Omega_1$. Consequently, for each $\omega \in \Omega_1$ we have $(k_L+l_L)(\Psi_n(\omega,t)) \leq (k_L+l_L)(\Upsilon(\omega,t))$ and $\lim_{n \rightarrow \infty} \Psi_n(\omega,t)=\Psi(\omega,t)$ for Lebesgue almost every $t \in [0,T]$.  Theorem \ref{pred_iff_integrable} guarantees 
that  there exists $\Omega_2 \subseteq \Omega$ with $P(\Omega_2)=1$ such that $m_L(\Upsilon(\omega,\cdot))<\infty$  for all $\omega \in \Omega_2$. Continuity of $k_L$ and $l_L$ at $0$, see Lemma \ref{le.properties_k_l}, and the classical version of Lebesgue's dominated convergence theorem implies that for all $\omega \in \Omega_1 \cap \Omega_2$ we have
\begin{align*}
    &\lim_{m,n \rightarrow \infty}m_L(\Psi_m(\omega,\cdot)-\Psi_n(\omega,\cdot))\\
    &\qquad= \lim_{m,n \rightarrow \infty}\Bigg( \int_0^T k_L(\Psi_m(\omega,t)-\Psi_n(\omega,t))+l_L(\Psi_m(\omega,t)-\Psi_n(\omega,t))\, {\rm d}t\\
    &\qquad \qquad \qquad \qquad + \int_0^T \norm{\Psi_m(\omega,t)-\Psi_n(\omega,t)}_{L_2(G,H)}^2 \wedge 1 \, {\rm d}t \Bigg)=0.
\end{align*}
Hence, for each $\omega \in \Omega_1 \cap \Omega_2$ the sequence $(\Psi_n)_{n \in \mathbb{N}}$ is Cauchy in the modular topology, which by Lemma \ref{le.completeness_modular_top} and the fact that $\Psi_n(\omega)\rightarrow \Psi(\omega)$ for Lebesgue a.a.\ $t \in [0,T]$ allows us to conclude that $\Psi(\omega)\in \mathcal{M}_{{\rm det},L}^{\rm HS}$. Since $P(\Omega_1 \cap \Omega_2)=1$, Theorem \ref{pred_iff_integrable} shows $\Psi \in \mathcal{I}_{{\rm prd},L}^{\rm HS}$. Another application of Lebesgue's dominated convergence theorem establishes that  $\lim_{n \rightarrow \infty}\normm{\Psi_n-\Psi}_L=0$, which completes the proof by an application of Corollary \ref{co.small-if-small}.
\end{proof}

\bibliographystyle{plain}

\end{document}